\documentclass[10pt,reqno]{amsart}

 \usepackage{amsmath,amsfonts,amssymb,amscd,amsthm,amsbsy,bbm, epsf,calc,enumerate,color,graphicx,verbatim,cite,import}
\usepackage{color}
\usepackage{datetime,appendix}
\usepackage{chapterbib,epstopdf}

\usepackage{float}

\usepackage{ifpdf}
\ifpdf
    \usepackage[colorlinks,citecolor=black,linkcolor=black]{hyperref}
\fi

\theoremstyle{plain}
\newtheorem{definition}{Definition}[section]
\newtheorem{theorem}{Theorem}[section]

\newtheorem{lemma}[theorem]{Lemma}
\newtheorem{corollary}[theorem]{Corollary}
\newtheorem{proposition}[theorem]{Proposition}

\theoremstyle{definition}

\newtheorem*{Claim}{Claim}

\theoremstyle{remark}                  

\newcommand{\ud}{\underline}

\definecolor{darkgreen}{rgb}{0,0.4,0}

\newcommand{\R}{\mathbb{R}}

\DeclareMathOperator*{\esssup}{ess\,sup}
\numberwithin{equation}{section}

\def \l {\left(}
\def\r {\right)}
\def\e {\epsilon}

\def\XXint#1#2#3{{\setbox0=\hbox{$#1{#2#3}{\int}$ }
\vcenter{\hbox{$#2#3$ }}\kern-.6\wd0}}

\title[Degenerate Equations with Drifts]{Regularity Properties of Degenerate Diffusion Equations with Drifts}
\author[Inwon Kim and Yuming Paul Zhang]{\bfseries Inwon Kim and  Yuming Paul Zhang}

\address{
(I. Kim) Department of Mathematics \\ 
University of California   \\ 
Los Angeles\\
USA}
\email{ikim@math.ucla.edu}

\address{
(Y. Zhang) Department of Mathematics \\ 
University of California   \\ 
Los Angeles\\
USA}
\email{yzhangpaul@math.ucla.edu}

\begin{document}

\vspace{18mm} \setcounter{page}{1} \thispagestyle{empty}

\begin{abstract}
This paper considers a class of nonlinear, degenerate drift- diffusion equations. We study well-posedness and regularity properties of the solutions, with the goal to achieve uniform H\"{o}lder regularity in terms of $L^p$-bound on the drift vector field. A formal scaling argument yields that the threshold for such estimates is $p=d$, while our estimates are for $p>d+\frac{4}{d+2}$. On the other hand we are able to show by a series of examples that one needs $p>d$ for such estimates, even for divergence free drift.

\end{abstract}

\maketitle

\section{Introduction}

Let $u=u(x,t)$ be a nonnegative function which solves the following problem:
\begin{equation}\label{main}
    u_t=\Delta u^m+\nabla\cdot(u V)\quad \text{ in }\mathbb{R}^d\times[0,\infty) \quad \hbox{ with } m>1.
\end{equation}
The drift term $V:\R^d\to \R^d$ is assumed to be time-independent, though our results extend to $V(x,t) \in L^{\infty}(L^p(\R^d); \R^+)$.  

\medskip

The $m>1$ in the nonlinear diffusion term above represents anti-congestion effect, and has been considered in many physical applications, including fluids in porous medium and population dynamics. Our system \eqref{main} can be thus naturally contextualized as a population moving with preferences or fluids in a porous medium moving with wind (see e.g. \cite{hw95, density, carrillo, chayes, kimlei, mnr2013}). The goal of this paper is to investigate well-posedness and regularity properties of \eqref{main} in terms of bounds of $V$ in $L^p(\R^d)$.




\medskip

When $m=1$, our equation is the classical drift-diffusion equation where an extensive literature is available for the corresponding regularity results, as we will discuss below. When $V=0$, \eqref{main} is the classical {\it porous medium equation} (see the book \cite{vazquez}) where initially integrable, nonnegative weak solutions exist, is unique and immediately become H\"{o}lder continuous for positive times. In contrast to these two cases,  few regularity results are available for \eqref{main} with $m>1$ and nonzero $V$, even in smooth settings. Below we discuss differences in local behaviors of solutions between our equation and the aformentioned cases by a scaling argument.

\medskip

For given $a,r>0$, let $u_{a,r}(x,t):=au(rx,r^2 a^{m-1}t)$. Then $\tilde{u}:=u_{a,r}$ solves
\begin{equation*}\label{rescale}\partial_t \tilde{u}=\Delta \tilde{u}^m+ \nabla\cdot (\tilde{V} \tilde{u})\;\text{ with } \tilde{V}(x):=a^{m-1}rV(rx).\end{equation*}
When $V=0$, the above scaling was used in \cite{dibenedetto1983continuity, dibenedetto2} along with Di Giorgi-Nash-Moser iteration arguments to derive H\"{o}lder continuity results. Here $1/a$ is chosen to be  the size of oscillation for $u_{a,r}$ in the unit neighborhood, and our goal is to show that this oscillation decays with a polynomial rate as $r\to 0$.  Thus our interest is in the case when the oscillation is large, i.e. when  $a\leq r^{-\e}$ for arbitrary small $\e>0$.  Note that 

$$
\|\tilde{V}(\cdot)\|_{L^p(\R^d)}=a^{m-1} r^{1-\frac{d}{p}}\|V(\cdot)\|_{L^p(\R^d)}.
$$

Recalling that $a$ is bounded by an arbitrarily small negative power of $r>0$, it is plausible that if $V$ is bounded in $L^p(\R^d)$ for some $p>d$, then solutions to \eqref{main} behave like the classical porous medium equation in small scales and generate bounded, H\"{o}lder continuous solutions.  Indeed when $V \in L^p(\R^d)$ with $p>d$  we will show that weak solutions exist and stay bounded for all times, if the solutions are initially bounded.

\medskip

These heuristics however pose serious challenges to deliver uniform regularity results for our equation. The most apparent difference from the linear case comes from the fact that our diffusion is degenerate at low densities. Due to this degeneracy, the proof of oscillation reduction in \cite{dibenedetto1983continuity, dibenedetto2} already differs from the standard ones.  For us this step corresponds to Proposition~\ref{prop1} and Proposition~\ref{prop2}, which turns out to be more challenging due to the competition between the singularity of the drift and the degeneracy of the diffusion in small scales. Indeed for this reason we can only show the uniform H\"{o}lder continuity of solutions when $p>d+\frac{4}{d+2}$. Whether the results extend to the remaining range $d<p\leq d+\frac{4}{d+2}$ or not remains as an interesting open question. We expect that it possibly requires new ideas to extend the ``oscillation reduction" mentioned above up to $p=d$. Perhaps for the same reason it stays also open to show that solutions become immediately bounded when starting with merely  integrable initial data, when $p>d$. 

\medskip

   On the other hand we are able to show that when $p\leq d$, uniform H\"{o}lder estimates are impossible even among divergence-free vector fields, thus establishing half of the sharp threshold. This is again expected to hold from the above heuristics, however the corresponding result does not seem to be shown for the linear case $m=1$ to the best of our knowledge. Our proof, based on barrier arguments akin to \cite{loss}, uses the degeneracy of diffusion at low densities and thus cannot be extended to the linear case.

\medskip




\medskip





Below we state two theorems that summarizes our main results.

\medskip

\begin{theorem}[Well-posedness and regularity]
Let us consider \eqref{main} with nonnegative initial data $u_0\in L^1(\R^d)\cap L^{\infty}(\R^d)$ and with $\|V\|_{L^p(\mathbb{R}^d)}<\infty$.
\begin{itemize}
\item[(a)]{\rm [Theorem~\ref{uniformb}] } If  $p>d$, then there exists a weak solution $u\in C([0,\infty), L^1(\R^d))$. Moreover
$u$ is uniformly bounded for all $t\in [0,\infty)$,
$$
\sup |u| \leq C(m,d,\|u_0\|_{\infty},\|u_0\|_1,p,\|V\|_p).
$$
\item[(b)]{\rm [Theorem~\ref{unique:local}]} The weak solution is unique if $V$ is uniformly $C^1$ in $\R^d$. \\
\item[(c)]{\rm [Theoreml~\ref{thmcont}]} If $p> d+\frac{4}{d+2}$, and if $u$ is a weak solution of \eqref{main} in $\{|x|\leq 1\}\times [0,1]$ that is also bounded, then $u$ is H\"{o}lder continuous in $ \{|x|<1\} \times (0,1)$. 
\end{itemize}
\end{theorem}

\medskip

As for (b), when $V$ is not $C^1$, general uniqueness of weak solutions are open except between strong solutions: see Theorem~\ref{unique:local}. 

\medskip

Regarding (c),  the only relevant result for \eqref{main} that we are aware of is from \cite{chkk}, where integrability conditions are assumed on both $V$ and $\nabla V$. Let us also very briefly mention some results for the linear case $m=1$ where the threshold $L^\infty_tL^d_x$ remains the same. In \cite{Friedlander2011,Sverak} it is shown that if $V\in L^\infty([0,T],BMO^{-1}(\R^d))$, then an initially integrable solution becomes immediately H\"{o}lder continuous. Let us mention that $BMO^{-1}(\R^d)$ shares the same scaling property with $L^d(\R^d)$, however the corresponding result for $L^d(\R^d)$ drifts is open except for the stationary case, see \cite{Safonov}. In two dimensions, even $L^1$-bound for time independent divergence-free drift turns out to be sufficient to yield continuous solutions (\cite{Sverak},\cite{loss}). Corresponding regularity results for $m>1$ in two dimensions remains open.

\medskip
  
Next we state the singularity results for the threshold case, where $V\in L^d(\R^d)$.

 \begin{theorem}[Loss of regularity]
  There exist sequences of vector fields $\{V^i_n\}_n$, $i=1,2$, which are uniformly bounded in $L^d(\R^d)$, along with sequences of compactly supported, uniformly bounded initial data $\{u^i_{0,n}\}$, $i=1,2$, such that the following holds:
    \medskip
  
  \begin{itemize}
  \item[(a)]{\rm [Theorem~\ref{thm5.2}]} The solutions  $\{u_n^1\}_n$ of \eqref{main} with $V=V^1_n$ and initial data $u_{0,n}^1$ satisfies \\$\lim_{T\to\infty} \sup_{[0,T] }|u^1_n| =\infty$; \\
\item[(b)]{\rm [Theorem~\ref{thm:lackofcont}] }The solutions $\{u^2_n\}_n$ of \eqref{main} with $V=V^2_n$ and initial data $u_{0,n}^2$ stays uniformly bounded, but  they do not share any common mode of continuity.
\end{itemize}
\end{theorem}

\medskip

The sequence of drifts given in above theorem represents strongly compressive drifts concentrated near the origin. Thus one naturally asks whether the regularity of solutions are better with singular, but divergence-free drifts. It turns out that the critical norm for drifts stays the same for divergence-free drifts.

\begin{theorem} (Loss of regularity II)  {\rm [Theorem~\ref{divfree2} for $d=2$ and Theorem~\ref{divfree} for $d=3$]}\\
There is a sequence of vector fields $\{V_n\}_n$ that are uniformly bounded in $L^d(\R^d)$, and a sequence of uniformly smooth initial data $\{u_{n,0}\}_n$, such that the corresponding solutions $\{u_n\}$  of \eqref{main} are uniformly bounded in height but not bounded in any H\"{o}lder norm in a unit parabolic neighborhood. 
 \end{theorem}

We believe that above statement holds in general dimensions. To illustrate this point, we give examples both in dimensions two and three. The construction for both cases are similar, but the vector field we choose for three dimensions has more complex singularity structure than the other one. We  suspect that this is only due to increased technicality in computations.

\medskip

 The proof of above theorem is motivated by the corresponding result in \cite{loss}, where loss of continuity is shown for solutions of fractional diffusion with drift at critical regime. In contrast to \cite{loss} our example makes use of the degeneracy of diffusion in small density region, such as finite propagation properties or slow decay rate for the density heights: see section 5 for further discussions. For linear diffusion the corresponding loss of H\"{o}lder regularity results appear to be open, to the best of our knowledge. Let us mention that for linear diffusion with $L^1$-drifts, \cite{Sverak} shows the existence of discontinuous solutions for $d=3$, while in two dimensions time-dependent vector fields are needed to generate discontinuity in solutions (see \cite{loss}).

  \medskip

\bigskip

\textbf{Outline of the paper}

\medskip

Section 2 contains preliminary definitions and notations. 
Section 3 deals with  a priori estimate of solutions that yields existence and uniqueness of uniformly bounded weak solutions for $V\in L^p(\R^d)$ with $p>d$. In section 4 bounded solutions in the parabolic cylinder are shown to be H\"{o}lder continuous for drifts bounded in $L^{p}$ norms with $p>d+\frac{4}{d+2}$. The proof follows the strategy of DiBenedetto and Friedman \cite{dibenedetto1,dibenedetto2}.  The key idea there is to circumvent the low regularity near small densities to work with De Giorgi-Nash-Moser type iteration but with re-scaled cylinders, where its size depends on the oscillation of solutions, based on the scale invariance of the first two terms in \eqref{main} discussed above (see \eqref{alpha:rescale}). Our challenge when doing this is to carefully study how the singularity of the drift term affects the diffusion, especially in small density region. The original argument carries out without too much effort when $p>d+2$, and with more subtle arguments for $p> d+\frac{4}{d+2}$ (see Proposition~\ref{prop1} and Lemma~\ref{twoseq}).  Regularity property of solutions for $d<p\leq d+\frac{4}{d+2}$ stays open at the moment.
In section 5, we give several examples that illustrate the loss of regularity when the drifts are only bounded in $L^{d}(\mathbb{R}^d)$. We discuss potential vector fields as well as divergence free vector fields. 

\leavevmode
\medskip

\textbf{Acknowledgements.}
Both authors are partially supported by NSF grant DMS-1566578. We would like to thank Michael Hitrik, Kyungkeun Kang, Luis Silvestre and Monica Visan for helpful discussions and suggestions.


\medskip



\medskip

\section{Preliminaries and Notations}\label{defassump}

\begin{definition}\label{def1.1}
Let $u_0(x)\in L^\infty(\mathbb{R}^d)\cap L^1(\mathbb{R}^d)$ be non-negative. We say that a non-negative function $u(x,t):\mathbb{R}^d\times [0,T]\to[0,\infty)$ is a subsolution (resp. supersolution) to \eqref{main} if \begin{equation}\label{definitionsol}
\begin{aligned}
&\quad u\in C([0,T],L^1(\mathbb{R}^d))\cap L^\infty(\mathbb{R}^d\times[0,T]),\; \\
&uV\in L^2([0,T]\times\mathbb{R}^d)\;\text{ and }\; u^m\in L^2(0,T, \dot{H}^1(\mathbb{R}^d)).\end{aligned}
 \end{equation}
And for all non-negative test functions $\phi\in C_c^\infty(\mathbb{R}^d\times [0,T))$
\[\int_0^T\int_{\mathbb{R}^d} u\phi_tdxdt\geq (\text{resp. } \leq) \int_{\mathbb{R}^d} u_0(x)\phi(0,x)dx+\int_0^T\int_{\mathbb{R}^d}(\nabla u^m+uV)\nabla\phi\; dxdt.\]

We say $u$ is a weak solution to \eqref{main} if it is both sub- and supersolution of \eqref{main}, or equivalently, 
it satisfies for all test function $\phi\in C^\infty_c( \mathbb{R}^d\times [0,T))$, 
\[\int_0^T\int_{\mathbb{R}^d} u\phi_tdxdt=\int_{\mathbb{R}^d} u_0(x)\phi(0,x)dx+\int_0^T\int_{\mathbb{R}^d}(\nabla u^m+uV)\nabla\phi\; dxdt.\]
\end{definition}


\begin{definition}\label{condV}We say an integrable vector vector field $V:\R^d\to \R^d$ is {\it admissable} if $V=V_1+V_2$ where
$$\|V_1 \|_{\infty} +\|V_2\|_p<\infty \text{ for some $p>d$}.
$$
\end{definition}

\begin{definition}\label{logq}
We say $V:\R^d\to \R^d$ is bounded in $L^{p}_{\log^q}(\R^d)$ for $p,q>0$ if
\[\left\|V\right\|^p_{L^{p}_{\log^q}}:=\int_{\mathbb{R}^d}|V|^p\max\{\log^q |V|,1\}dx<\infty.\]
\end{definition}

\medskip

The following will be used to obtain regularity estimates in section 3.

\begin{lemma}\label{gagliardo}
(Gagliardo-Nirenberg Interpolation Inequality)\quad For any $u$ such that $u\in L^q(\mathbb{R}^d)$ and $\nabla u\in L^r(\mathbb{R}^d)$, there exists a constant $C(\alpha,r,q,s)$ that
\[\left\||\nabla|^s u\right\|_p\leq C\left\|\nabla u\right\|_r^\alpha\left\|u\right\|_q^{1-\alpha}\]
where we require
\begin{equation}\label{condition1}
\frac{1}{p}=\frac{s}{d}+\left(\frac{1}{r}-\frac{1}{d}\right)\alpha+\frac{1-\alpha}{q} \end{equation}
\begin{equation}\label{condition2}
\text{ and }s\leq\alpha<1, \, r\geq 1, \, q\geq 1.
\end{equation}
For functions $u:B_R\to\mathbb{R}$, the interpolation inequality has the same hypotheses as above and reads
\[\left\||\nabla|^s u\right\|_p\leq C_1\left\|\nabla u\right\|_r^\alpha\left\|u\right\|_q^{1-\alpha}+C_2\|u\|_1.\]
where the constants $C_1, C_2$ are independent of $R$ for all $R$ large enough.
\end{lemma}
We refer readers to \cite{nirenberg} for the proof.



\medskip



\medskip

\textbf{Notations.}

\medskip

$\circ$ Given $S\subset \mathbb{R}^d$ (or $\mathbb{R}^{d+1}$) measurable, we write $|S|$ to be the Lebesgue measure of $S$ in $ \mathbb{R}^d$ (or $\mathbb{R}^{d+1}$). We write $B_r(x) \subset\mathbb{R}^d$ as a ball centered at $x$ with radius $r$, and  denote $B_r=B_r(0)$.
\medskip

$\circ$ For simplicity we denote
$$\|\cdot\|_p:=\|\cdot\|_{L^p(\R^d)}\hbox{ and }\|\cdot\|_\alpha:= \|\cdot\|_{C^{\alpha}(\R^d)} \text{ and }
$$
\[ osc_S(u):=\sup_{x\in S}u-\inf_{x\in S}u,\]
for any measurable function $u:\mathbb{R}^d \to \R$ and  $S\subset \mathbb{R}^d$. 
\medskip

$\circ$ The scaled parabolic cylinders are denoted by
\begin{equation}\label{Qrc}Q(r,c):=\{x,|x|<r\}\times (-c{r^2},0) \hbox{ for } r,c>0. \end{equation}
The standard parabolic cylinder is denoted by $Q_r:=Q(r,1)$.
\medskip

\medskip

$\circ$ Throughout this paper, the constant $C$ represents  {\it universal constants}, by which we mean various constants that only depends on $m,d$ and $L^1,L^\infty$ norms of the initial data $u_0$. In addition, $C$  may also depend on $\|V\|_p$ or $\|V\|_{L^p_{loc}}$ with $p$ given in the statement of the Theorem.  We may write $C(A)$ or $C_A$ to emphasize the dependence of $C$ on $A$.

\medskip

$\circ$ We write 
$A\lesssim B$
if $A\leq CB$ for some universal constant $C$. When we write $A\lesssim_D B$, we mean $A\leq CB$ where $C$ depends on universal constants and $D$ (with particular emphasis on the dependence of $D$). By $A\sim B$, we mean both $A\lesssim B$ and $B\lesssim A$.

\medskip



\section{Priori Estimates}

In this section several a priori estimates are obtained for solutions for \eqref{main}. 

\medskip

Let $V$ be an admissible vector field given in Definition \ref{condV}.
For any $\e>0$, consider smooth vector fields $\{V_1^\epsilon,V_2^\epsilon\}$ such that, as $\e\to 0$, $V_1^\epsilon$ converges to $V_1$ in $L^\infty(\R^d)$ and $V_2^\epsilon$ converges to $V_2$ in $L^p(\R^d)$. Denote $V^\epsilon:=V_1^\epsilon +V_2^\epsilon$ and
 \[\varphi_\epsilon(x):=x^m+\epsilon x.\]
For some large $r>0$, we consider $u_{\epsilon,r}$ which solves the following problem:
\begin{equation}\label{approxeqn}
\left\{\begin{aligned}
&\frac{\partial}{\partial t}u_{\epsilon,r}=\Delta \varphi_\epsilon(u_{\epsilon,r})+\nabla\cdot (u_{\epsilon,r} V^\epsilon)= 0 &\text{ in }B_r\times [0,T], \\
&(\nabla \varphi_\epsilon(u_{\epsilon,R})+ u_{\epsilon,r} V^\epsilon)\cdot\nu=0 &\text{ on }(\partial B_r )\times [0,T],\\
& u_\epsilon(x,0)=u_{0}(x) &\text{ on } B_r
\end{aligned}
\right.
\end{equation}
where $\nu$ denotes the outward unit normal on $\partial B_r$. Note that \eqref{approxeqn} is a uniformly parabolic quasi-linear equation with smooth coefficients, and thus $u_{\epsilon,r}$ exists and is smooth.

In the following theorem, we are going to prove that $u_{\epsilon,r}$ are uniformly bounded independent of $\epsilon$ and $r$. We use a refined iteration method of Lemma 5.1 \cite{preventing}. 

\begin{theorem}\label{uniformb}
Let $u=u_{\epsilon,r}$ solves \eqref{approxeqn} with initial data $u_0\in L^1(\mathbb{R}^d)\cap L^\infty(\mathbb{R}^d)$ and admissible vector fields $V^\epsilon=V^\epsilon_1+V^\epsilon_2$. Then $u(x,t)$ is uniformly bounded for all $(x,t)\in \mathbb{R}^d\times [0,\infty)$. The bound only depends on $m,p,\|V^\epsilon_1\|_\infty,\|V^\epsilon_2\|_{p},\|u_0\|_1$ and $\|u_0\|_\infty$.
\end{theorem}

\begin{proof}
Without loss of generality, let us suppose that the total mass of $u_0$ is $1$ and so is the total mass of $u(\cdot,t)$ by the equation. Let us omit the script $\epsilon$ on $V^\epsilon$ and simply write $V=V_1+V_2$.

Denote $u_1:=\max\{(u-1),0\}$. Since $u$ is smooth, we multiply $u_1^{n-1}$ on both sides of \eqref{approxeqn} and find
\[\partial_t\int_{B_r} u_1^n dx=n
\int_{B_r} u_tu_1^{n-1}dx\leq
-mn\int_{B_r} u^{m-1}\nabla u \nabla u_1^{n-1}dx-n\int_{B_r} V u\nabla u_1^{n-1}dx.
\]
Since in the region where $\nabla u_1\ne 0$, $u\geq 1$, the above
\[
\leq
- c_m\int_{B_r}\left|\nabla u_1^{\frac{n}{2}}\right|^2dx- 2(n-1){\int_{B_r} Vuu_1^{\frac{n}{2}-1}\nabla u_1^{\frac{n}{2}}dx}.
\]
We have for any $\delta>0$,
\[n\left|\int_{B_r} Vuu_1^{\frac{n}{2}-1}\nabla u_1^{\frac{n}{2}}dx\right| \leq \delta \int_{B_r}\left|\nabla u_1^{\frac{n}{2}}\right|^2dx
+C{n^2}\int_{\mathbb{R}^d\cap\{u\geq 1\}} \left|Vuu_1^{\frac{n}{2}-1}\right|^2dx.\]
Later we will fix a $\delta$ small enough such that the sum of the positive coefficients in front of $\int_{B_r}|\nabla u_1^{\frac{n}{2}}|^2dx$ terms are bounded by $c_m$.
The above shows
\begin{equation}\label{noninter1}
\partial_t\int_{B_r} u_1^n dx\lesssim - \int_{B_r}\left|\nabla u_1^{\frac{n}{2}}\right|^2dx+ n^2\underbrace{\int_{\{u\geq 1\}} \left|Vuu_1^{\frac{n}{2}-1}\right|^2dx}_{X_n:=}
\end{equation}
where the constant in $``\lesssim"$ depends only on $m,\delta$. Next
\[X_n\lesssim \underbrace{\int_{\{u\geq 1\}} \left|V_1(1+u_1)u_1^{\frac{n}{2}-1}\right|^2dx}_{X_{n1}:=}+\underbrace{\int_{\{u\geq 1\}} \left|V_2(1+u_1)u_1^{\frac{n}{2}-1}\right|^2dx}_{X_{n2}}\]
and
\[X_{n1}\lesssim \int_{\{u\geq 1\}} \left|u_1^{n-1}+u_1^{n}\right|^2dx.\]
By H\"{o}lder's inequality,
$$X_{n2}\lesssim \l \int_{\{u\geq 1\}} V_2^{2q_1}dx\r^{\frac{1}{q_1}}\l\int_{\{u\geq 1\}} u_1^{nq_2}+u_1^{(n-2)q_2}dx\r^{\frac{1}{q_2}}\lesssim \l\int_{\{u\geq 1\}} u_1^{nq_2}+u_1^{(n-2)q_2}dx\r^{\frac{1}{q_2}}$$
where $q_1=\frac{p}{2},\frac{1}{q_1}+\frac{1}{q_2}=1$. By the condition
\begin{equation}\label{2dcond}
1>\frac{1}{q_2}>1-\frac{2}{d}.
\end{equation}
Because $u$ has total mass $1$, the total volume of the set $\{u\geq 1\}$ is bounded by $1$.
So $X_{n1}\lesssim X_{n2}$ and we have
$$X_n \lesssim \l\int_{\{u\geq 1\}} u_1^{nq_2}+u_1^{(n-2)q_2}dx\r^{\frac{1}{q_2}}\lesssim\l
\int_{\{u\geq 1\}} u_1^{nq_2}+1dx\r^{\frac{1}{q_2}}
\lesssim
\left\|u_1^\frac{n}{2}\right\|^2_{2q_2}+1
$$

By Gagliardo-Nirenberg inequality,
\[\left\|u_1^\frac{n}{2}\right\|_{2q_2}\leq C_1\left\|\nabla u_1^\frac{n}{2}\right\|^\gamma_2\left\|u_1^\frac{n}{2}\right\|_1^{1-\gamma}+C_1\left\|u_1^\frac{n}{2}\right\|_1\]
where
$\frac{1}{2 q_2}=\l\frac{1}{2}-\frac{1}{d}\r\gamma+(1-\gamma)$
and
\[\gamma=\l1-\frac{1}{2q_2}\r\Big/\l\frac{1}{2}+\frac{1}{d},\r\]
which belongs to $(0,1)$ due to \eqref{2dcond}, and $C_1$ only depends on $p$.
 By Young's inequality
\begin{equation}\label{noninter2}
X_n\leq \frac{\delta}{n^2}\left\|\nabla u_1^\frac{n}{2}\right\|^{2}+C_\delta n^{c_\gamma}\l\int u_1^\frac{n}{2}dx\r^2+C
\end{equation}
with $c_\gamma=\frac{2\gamma}{1-\gamma}$.

Again using Galiardo-Nirenberg inequality and Young's inequality it follows 
\[\left\|u_1^\frac{n}{2}\right\|_{2}\lesssim \left\|\nabla u_1^\frac{n}{2}\right\|^\beta_2\left\|u_1^\frac{n}{2}\right\|_1^{1-\beta}+\left\|u_1^\frac{n}{2}\right\|_1\lesssim
\left\|\nabla u_1^\frac{n}{2}\right\|_2+\left\|u_1^\frac{n}{2}\right\|_1 \]
with $\beta=\frac{1}{2}/\l\frac{1}{2}+\frac{1}{d}\r$. So for some universal $C,c>0$
\begin{equation}\label{noninter3}
\left\|\nabla u_1^\frac{n}{2}\right\|^2_2\geq C\left\|u_1^\frac{n}{2}\right\|_2^2-c\left\|u_1^\frac{n}{2}\right\|_{1}^2.
\end{equation}

From \eqref{noninter1}, \eqref{noninter2} and \eqref{noninter3}, we have
\[\partial_t \int_{B_r} u_1^n +c_0 \int_{B_r} u_1^ndx \leq C n^{c_\gamma+2}\l\int_{B_r} u_1^\frac{n}{2}dx\r^2+Cn^2.
\]
Now let $n_k=2^k$ for $k=0,1,2...$ and $A_k(t)=\int u_1^{n_k}(x,t)dx$.  To conclude the proof we need the following lemma, whose proof will be given in the appendix.

\begin{lemma}\label{iteration}
Suppose $\{n_k\}$ is a sequence defined by 
\begin{equation}\label{sequence}
    n_0=1,\quad n_{k+1}:=2  n_k+a \text{ for all }k\geq 0, \hbox{ where } a>-1.
\end{equation}

Let $\{A_k(\cdot),k =0,1,...\}$ be a sequence of differentiable, positive functions on $[0,\infty)$ that satisfies 
\[\frac{d}{dt}A_k+C_0 A_k\leq C_1^{n_k}+{C_1}^k (A_{k-1})^{2+{C_1}n_k^{-1}},\]
for some constants $C_0,{C_1}$. Then $\{B_k(t):=A_k^{(n_k^{-1})}(t)\}$ are uniformly bounded for all $t> 0$ and $k$, given that $\{B_k(0)\}$ with respect to $k$ and $\{B_0(t)\}$ are uniformly bounded with respect to $t>0$.
\end{lemma}

\medskip

From above lemma, $A_k^{n_k^{-1}}$ are uniformly bounded. We have that $\|u^n_1(\cdot,t)\|_n$ are uniformly bounded for all $t$ and $n\in\{2^k,k=0,1,2...\}$. By interpolation, this shows that $\|u_1\|_p< \infty $ for $1\leq p \leq \infty$. Since
\[\int_{B_r} u^n dx\leq \int_{\{u\geq 2\}} (2u_1)^n dx+2^{n-1}\int_{\{u\leq 2\}} u\; dx\lesssim 2^n,\]
we find the $L^\infty$ bound of $u$ which is independent of $r,\epsilon$.
\end{proof}

\subsection{Existence}\label{wellp}

In this section, we show existence  of solutions to \eqref{main} with $V\in L^\infty(\R^d)+L^p(\R^d)$ for some $p>d$. 
\begin{theorem}\label{thm2.4}
Assume $V$ is admissible.
Then there exists a weak solution $u$ to \eqref{main} with nonnegative initial data $u_0 \in L^\infty(\mathbb{R}^d)\cap L^1(\mathbb{R}^d)$.
\end{theorem}
\begin{proof}
The proof is parallel to the previous works \cite{density, bertozzi2009, bedrossian2011}. Recall that $u_{\epsilon,r}$ solve \eqref{approxeqn}. Theorem \ref{uniformb} states that for all $t\in[0,T]$, $\{u_{\epsilon,r}\}$ are uniformly bounded in $L^1(B_r)\cap L^\infty (B_r)$ independent of  $\epsilon,r$. 

\medskip

Using $\varphi_\epsilon(u_{\epsilon,r})$ as the test function in \eqref{approxeqn}, we obtain
\[\left(\int_{B_r}\frac{1}{m+1}u_{\epsilon,r}^{m+1}+\frac{\epsilon}{2}u^2 dx\right)\Bigg|^{T}_{0}=-\iint_{B_r\times[0,T]}|\nabla \varphi_\epsilon(u_{\epsilon,r})|^2dxdt-\iint_{B_r\times[0,T]} u_{\epsilon,r} V^\epsilon \cdot \nabla \varphi_\epsilon(u_{\epsilon,r})dxdt. \]
From H\"{o}lder and Young's inequality
\begin{equation}\label{Aepsilon}
\iint_{B_r\times[0,T]}|\nabla \varphi_\epsilon(u_{\epsilon,r})|^2dxdt\leq C+\iint_{B_r\times[0,T]} u_{\epsilon,r}^2 |V^\epsilon|^2dxdt
\end{equation}
Let $q$ be such that $\frac{2}{q}+\frac{2}{p}=1$. Then 
$$
\left\|u_{\epsilon,r} V^\epsilon\right\|^2_{L^2(B_r\times [0,T])}\leq \left\|u_{\epsilon,r} V_1^\epsilon\right\|^2_{L^2(B_r\times [0,T])}
+2 
\left\|u_{\epsilon,r} \right\|^2_{L^q(B_r\times [0,T])}
\left\|V_2^\epsilon \right\|^2_{L^p(B_r\times [0,T])}.
$$
The two terms on the right hand side are uniformly bounded with respect to $\e$ and $r$, since $\{u_{\epsilon,r}\}$ are uniformly bounded in $L^\infty(B_R)\cap L^1(B_R)$.

By \eqref{Aepsilon}, $\{\nabla \varphi_\epsilon(u_{\epsilon,r})\}$ are uniformly bounded in $L^2(B_r\times [0,T])$.  As in Theorem 1 of \cite{bedrossian2011},  $\{u_{\epsilon,r}\}_{\epsilon>0}$ is precompact in $L^1(B_r\times [0,T])$. Along a subsequence as $\epsilon\to 0$, we obtain a weak solution $u_r$ to \eqref{main} in $B_r\times [0,T]$ with no-flux boundary condition. 
Then following the proof of Theorem 1 \cite{bedrossian2011}, it follows that
$u_r\in C([0,T], L^1(B_R))$.

Now we send $r\to\infty$. Notice that the $L^\infty([0,T],L^p(B_r)), p\in [1,\infty]$ bounds we have on $\{u_r\}$ and $L^2(B_r\times [0,T])$ bounds on $|\nabla u_r^m|$ are independent of $r$. These bounds yields sufficient compactness to yield a subsequential limit $u\in C([0,T], L^1(\mathbb{R}^d))$ which is a weak solution of \eqref{main}. For complete details, we refer to Theorem 2 \cite{bedrossian2011}.

\end{proof}

\subsection{Uniqueness}

This section discusses two uniqueness results.  First let us consider a relatively smooth vector field $V$ and show comparison principle for weak solutions.

\begin{theorem}\label{thm:comp}
Write $V=(V^i)_{i=1,...,d}$ and $I_d$ as $d\times d$ identity matrix. Suppose for some $M>0$
\begin{equation}\label{cond:V} |V|<+\infty,\; -M I_d\leq  DV \leq  M I_d.\end{equation}
Let $\bar{u},\underline{u}$ be respectively a subsolution and a supersolution of \eqref{main} with initial functions $\bar{u}_0,\underline{u}_0$ such that $\bar{u}_0\leq\underline{u}_0$. Then $\bar{u} \leq\underline{u}$ for $t\geq 0$.
\end{theorem}
\begin{proof}
Define $a(x,t)=(\underline{u}^m-\bar{u}^m)/(\underline{u}-\bar{u})$. 
Suppose $\epsilon>0$ is small enough and $N$ is large enough such that
\begin{equation}\label{cond:aNe2} 
 \iint_{\mathbb{R}^d\times [0,1]\cap \{a\geq N\}}|\underline{u}^m-\bar{u}^m| dxdt\leq \epsilon^2.
\end{equation}
Let $a_{N,\epsilon}$ be a smooth approximation of $a+\epsilon$ such that for $t\in [0,1]$
\begin{equation}\label{cond:aNe} \epsilon\leq a_{N,\epsilon}\leq N,\quad \|a_{N,\epsilon}(\cdot,t)-\min\{a(\cdot,t),N\}-\epsilon\|_2\leq \epsilon. \end{equation}

For any smooth non-negative compactly supported test function $\xi$, we consider the following dual problem to \eqref{main}:
\begin{equation}\label{eqn:dual}
\left\{\begin{aligned}
&\varphi_t+a_{N,\epsilon}\Delta \varphi-V\cdot\nabla \varphi +\xi= 0 &\text{ in }\mathbb{R}^d\times [0,T]; \\
& u(x,T)=0 &\text{ on } \mathbb{R}^d
\end{aligned}
\right.
\end{equation}
for some $T\in (0,1]$ to be determined.
Since $a_{N,\epsilon}\geq\epsilon$, there is a unique solution $\varphi \geq 0$ of \eqref{eqn:dual} which is smooth.

We write $u=\underline{u}-\bar{u}$. Since $\underline{u}$ and $\bar{u}$ are respectively super and subsolutions,
by the weak inequality satisfied by $u$ with respect to test function $\varphi$, we deduce
\[0\leq \iint_{\mathbb{R}^d\times [0,T]}u\varphi_t dxdt+ \iint_{\mathbb{R}^d\times [0,T]} au\Delta \varphi dxdt- \iint_{\mathbb{R}^d\times [0,T]}uV\nabla\varphi dxdt-\int_{\mathbb{R}^d}u(x,0)\varphi(x,0)dx.\]
Using that $u(\cdot,0)\geq 0, \varphi\geq 0$ and \eqref{eqn:dual}, then
\[\iint_{\mathbb{R}^d\times [0,T]}u\xi dxdt\leq \iint_{\mathbb{R}^d\times [0,T]}|u||a-a_{N,\epsilon}||\Delta\varphi|dxdt\]
\begin{equation}\label{ineq:aNe}
\leq \left(\iint_{\mathbb{R}^d\times [0,T]}a_{N,\epsilon} |\Delta\varphi|^2 dxdt\right)^\frac{1}{2}\left(\iint_{\mathbb{R}^d\times [0,T]}\frac{|a-a_{N,\epsilon}|^2}{a_{N,\epsilon}}|u|^2 dxdt\right)^\frac{1}{2}.
\end{equation}

We want to obtain a priori estimate for the term $\Delta\varphi$.

\medskip

Fix $\zeta(t)$ be a smooth function such that $1\leq \zeta(t)\leq 2$ and $\zeta_t\geq 2dM+4M+1$ for $t\in [0,T]$ which can be done when $T$ is small enough. 

We multiply \eqref{eqn:dual} by $\zeta\Delta\varphi$, after integration that is
\[\iint_{\mathbb{R}^d\times [0,T]}\zeta V\cdot\nabla\varphi\Delta\varphi dxdt=\]
\[\iint_{\mathbb{R}^d\times [0,T]}\varphi_t\zeta\Delta\varphi dxdt+\iint_{\mathbb{R}^d\times [0,T]}\zeta a_{N,\epsilon}|\Delta\varphi|^2 dxdt+\iint_{\mathbb{R}^d\times [0,T]}\zeta\xi\Delta\varphi dxdt.\]
Using integration by parts and H\"{o}lder's inequality in the first inequality, the above (see Theorem 6.5 \cite{vazquez} in 6.2.1 for details).
\[\geq \iint_{\mathbb{R}^d\times [0,T]}\frac{1}{2}\zeta_t|\nabla\varphi|^2 dxdt+\iint_{\mathbb{R}^d\times [0,T]} \zeta a_{N,\epsilon}|\Delta\varphi|^2 dxdt-\iint_{\mathbb{R}^d\times [0,T]}\zeta \nabla\xi\nabla\varphi dxdt\]
\[\geq (dM+2M)\iint_{\mathbb{R}^d\times [0,T]}|\nabla\varphi|^2 dxdt+\iint_{\mathbb{R}^d\times [0,T]} a_{N,\epsilon}|\Delta\varphi|^2 dxdt-C\iint_{\mathbb{R}^d\times [0,T]}|\nabla\xi|^2 dxdt.\]
Then $\iint_{\mathbb{R}^d\times [0,T]} a_{N,\epsilon}|\Delta\varphi|^2 dxdt\leq$
\begin{equation}\label{eq:testfunzeta}
 \iint_{\mathbb{R}^d\times [0,T]}\zeta V\cdot\nabla\varphi\Delta\varphi dxdt-(d+2)M\iint_{\mathbb{R}^d\times [0,T]}|\nabla\varphi|^2 dxdt+C\iint_{\mathbb{R}^d\times [0,T]}|\nabla\xi|^2 dxdt.
\end{equation}
By \eqref{cond:V}, $-(V_{x_j}^i)\leq M I_d$ and $|\nabla\cdot V|\leq dM$, 
\[\iint_{\mathbb{R}^d\times [0,T]}\zeta V\cdot\nabla\varphi\Delta\varphi dxdt=\frac{1}{2}\iint_{\mathbb{R}^d\times [0,T]}\zeta |\nabla \varphi|^2\nabla\cdot V dxdt -\iint_{\mathbb{R}^d\times [0,T]}\zeta \sum_{i,j}\varphi_{x_i}V^i_{x_j}\varphi_{x_j}dxdt\]
\[\leq ({d}+2)M\iint_{\mathbb{R}^d\times [0,T]} |\nabla\varphi|^2dxdt.\]
Plugging the above inequality and \eqref{eq:testfunzeta} into \eqref{ineq:aNe}, we get
\[\iint_{\mathbb{R}^d\times [0,T]} u\xi dx dt\leq C\|\nabla \xi\|_{L^2(\mathbb{R}^d\times [0,T])} \left(\iint_{\mathbb{R}^d\times [0,T]}\frac{|a-a_{N,\epsilon}|^2}{a_{N,\epsilon}}|u|^2 dxdt\right)^\frac{1}{2}.\]

Now we use \eqref{cond:aNe} and find out \[\iint_{\mathbb{R}^d\times [0,T]}|a-a_{N,\epsilon}|^2|u|^2
 \leq 2\iint_{\mathbb{R}^d\times [0,T]}{\left|\min\{a,N\}+\epsilon-a_{N,\epsilon}\right|^2}|u|^2 dxdt+\]\[2\iint_{\mathbb{R}^d\times [0,T]}\epsilon^2|u|^2 dxdt+2\iint_{\{a>N\}}a^2|u|^2 dxdt\]
\[\leq \left(2\|u\|_\infty^2+2\|u\|^2_{L^2(\mathbb{R}^d\times [0,T])}+2\right) \epsilon^2\leq C\epsilon^2.\]
So since $a_{N,\epsilon}\geq \epsilon$, by \eqref{ineq:aNe}
\[\iint_{\mathbb{R}^d\times [0,T]} u\xi dx dt\leq C\|\nabla \xi\|_{L^2(\mathbb{R}^d\times [0,T])} \epsilon^\frac{1}{2}.\]
Letting $\epsilon>0$, we conclude that $\iint_{\mathbb{R}^d\times [0,T]} u\xi dx dt\leq 0$ for all arbitrary $C_c^\infty$ test function $\xi\geq 0$. And so $u\leq 0$ within time $[0,T]$. 

{Finally since $T$ only depends on $d,M$, doing this repeatedly finishes the proof.}
\end{proof}

  \medskip

Our second uniqueness result is a consequence of the following $L^1$ contraction, which holds between ``strong solutions" if $m$ is not too large depending on the singularity of $V$. The existence of strong solutions remain open, with the exception of zero drift case (see \cite{AltLuckhaus} and section 8.1.1 of  \cite{vazquez}).
\begin{theorem}\label{unique:local}
Suppose $\|V\|_p<\infty$ for some $p\geq 2$ and $1<m< 1+\frac{2}{p}$. Let $u_1,u_2$ be two nonnegative weak solutions to \eqref{main} with initial datas $u_{1,0},u_{2,0}$ respectively. Assume in addition that 
\[\partial_t( u_1- u_2)\in L^1(\mathbb{R}^d\times [0,T]).\]
Then the following holds:
\[\int_{\mathbb{R}^d} (u_1-u_2)_+(t) dx\leq 
\int_{\mathbb{R}^d} (u_{1,0}-u_{2,0})_+ dx \quad\hbox{ for } 0\leq t\leq T.
\]
\end{theorem}

\begin{proof}

Let $\varphi\in C^1(\mathbb{R})$ be such that  $ \varphi(s)=0$ if $s\leq 0$ and $ \varphi(s)=1$ if $s\geq 1$, with $\varphi'(s)\in (0,2)$. Denote $\varphi_n(s):=\varphi(ns)$ for $n=1,2,..$. By definition of the weak solution we have, with $w:=u_1^m-u_2^m$,

\[
\iint_{\mathbb{R}^d\times [0,T]} (u_1-u_2)_t\varphi_n(w)dxdt=-\iint_{\mathbb{R}^d\times [0,T]}\nabla  w\nabla\varphi_n(w)dxdt-\iint_{\mathbb{R}^d\times [0,T]}(u_1-u_2)V\cdot \nabla \varphi(w) dxdt\]
\[=-\iint_{\mathbb{R}^d\times [0,T]} |\nabla w|^2 \varphi_n'(w)dxdt{-\iint_{\mathbb{R}^d\times [0,T]} (u_1-u_2) V\cdot\nabla w\; \varphi_n'(w)dxdt}.
\]
Since $\varphi_n'\leq 2n$,
\[{-\iint_{\mathbb{R}^d\times [0,T]} (u_1-u_2) V\nabla w \;\varphi_n'(w)dxdt}\leq  \iint_{\mathbb{R}^d\times [0,T]} |\nabla w|^2 \varphi_n'(w)dx+2n\iint_{0<w\leq \frac{1}{n}} |u_1-u_2|^2 |V|^2dxdt.\]
When  $p>2$, let $q$ be such that $\frac{2}{q}+\frac{2}{p}=1$. Note that when $w>0$, $u_1> u_2\geq 0$ and thus 
$$ |u_1-u_2|^m \leq |u_1^m - u_1^{m-1}u_2| \leq  |w|.
$$
Thus we have
\begin{align}\nonumber\iint_{0<w\leq \frac{1}{n}} |u_1-u_2|^2 |V|^2dxdt &\leq
\iint_{0<w\leq \frac{1}{n}} |w|^{\frac{1}{m}(2-\frac{2}{q})}|u_1-u_2|^\frac{2}{q} |V|^2dxdt.
\\&\leq n^{-\frac{1}{m}(2-\frac{2}{q})}\left(\iint_{0<w\leq \frac{1}{n}} |u_1-u_2|dxdt\right)^\frac{2}{q} \left(\iint_{\mathbb{R}^d\times [0,T]}|V_2|^p dxdt\right)^\frac{2}{p}.
\nonumber
\end{align}
Then, since $|u_1-u_1| \leq C([0,T];L^1(\R^d))$, it follows that
\[\int_{\mathbb{R}^d} (u_1-u_2)_t \;\varphi_n(w)dx\leq Cn^{1-\frac{1}{m}(2-\frac{2}{q})}=C n^{1-\frac{p+2}{pm}},\]
the right hand side of which goes to $0$ as $n\to \infty$ due to $m<1+\frac{2}{p}$. 
Now we send $n\to\infty$ to derive the desired inequality:
\[\int_{\mathbb{R}^d} (u_1-u_2)_+(t) dx\leq 
\int_{\mathbb{R}^d} (u_{1,0}-u_{2,0})_+ dx.\]

If $p=2$, parallel and easier proof yields the result. \end{proof}


\section{H\"{o}lder Continuity}\label{sec:holder}

\subsection{Interior Estimates}
In this section we establish the H\"{o}lder continuity results for \eqref{main}.

\begin{theorem}\label{thmcont}
 Suppose $V$ is locally uniformly bounded in $L^p(\R^d)$ for some $p>d+\frac{4}{d+2}$. Let $u$ be a non-negative weak solution to equation \eqref{main} in $Q_1$. If $u(\cdot,t)$ is uniformly bounded by $M$ in $Q_1$, then $u(\cdot,\cdot)$ is H\"{o}lder continuous in $Q_{\frac{1}{2}}$. The H\"{o}lder norm only depends on $M$, $m$, $p$, $d$ and $\|V\|_{L^p_{loc}}$.
 \end{theorem}

The proof of above theorem consists of several lemmas and propositions.  We begin with notations. For given $p$, we will use 
\[\delta_1:=2-\frac{2d}{p},\;\delta_2:=\frac{1}{2}\delta_1,\; q_1:=1-\frac{2}{p},\;q_2:=1-\frac{1}{p}.\] 
In particular if $p>d\geq 2$, $q_1+\frac{2}{d}>1,q_2+\frac{2}{d+2}>1$. Let us define a new variable

 \begin{equation}\label{pressure}
 \nu:=u^{\frac{1}{m}}.
 \end{equation} Then $\nu$ satisfies
\begin{equation}\label{eqcontv}\frac{\partial}{\partial t}\nu^{\frac{1}{m}}=\Delta \nu+\nabla\cdot(\nu^{\frac{1}{m}}{V}).\end{equation}

\medskip

Next we re-scale $\nu$ by 
\begin{equation}\label{alpha:rescale}
v(x,t):=\nu(rx,r^2 w^{-\alpha}t) \,\,\hbox{ in } Q(r,w^{-\alpha}) \quad \hbox{ with } \alpha := \frac{m-1}{m}.
\end{equation}
Then  $v$ solves
\begin{equation}\label{eqnu}
w^\alpha(v^\frac{1}{m})_t=\Delta v+r\nabla\cdot(v^\frac{1}{m}\tilde{V}),  \quad \hbox{ where } \tilde{V}(x,t):=V(rx,r^2w^{-\alpha}t).\end{equation}
Also denote
\[v_k^+:=\max\{(v-k),0\},\quad v_k^-:=\max\{(k-v),0\}. \]
 
 \medskip
 
 We begin with an energy inequality. The proof of the lemma below are in the same spirit of the ones in Theorem 1.2 in \cite{dibenedetto2} and Lemma 6.5 \cite{dibenedetto1} which applies to \eqref{main} with $V=0$. We will emphasize on the differences in the proof that occurs due to  the nonzero drift term.

\begin{lemma}\label{inte}
Suppose $v$ satisfies \eqref{eqnu} in a neighbourhood of $Q_1$ for some positive $w,r$ such that  $w\geq osc_{Q_1}v $. Suppose $V$ is locally uniformly bounded in $L^p(\mathbb{R}^d)$ for some $p>0$. Let $\zeta\in C_0^\infty(Q_1)$ be non-negative and 
\[\zeta\leq 1, \quad|\nabla\zeta|\leq C_1, \quad|\Delta\zeta^2|\leq C_1^2, \quad|\zeta_t|\leq C_2.\]
Denote $B':=B_1\cap supp\{\zeta\}$ and for $q\in (0,1]$
\[B_{k;q}:=\left(\int_{-1}^0\left(\int_{B'} \chi_{\{v(x,t)< k\}} dx\right)^qdt\right)^{\frac{1}{q}},\]\[ A_{k;q}:=\left(\int_{-1}^0\left(\int_{B'} \chi_{\{v(x,t)> k\}} dx\right)^qdt\right)^{\frac{1}{q}}\]
and $M^+,M^-$ as the supremum and infimum of $v$ in $Q_1$ respectively. 

If $\frac{w}{4}\geq M^-$, then for $t\in[-1,0], k\leq M^+$, 
\[\int_{B_1\times\{t\}}|{v_k^-}\zeta|^2 dx+\int_{-1}^t\left\|\nabla \l v_k^-\zeta\r\right\|^2_{2,B_1\times \{s\}}ds \lesssim (C_1^2+C_2) w^{2} B_{k;1}+\]
\begin{equation*}
r^{\delta_1}w^{\frac{2}{m}}B_{k;q_1}^{q_1}+
C_1 r^{\delta_2}w^{1+\frac{1}{m}} B_{k;q_2}^{q_2}.
\end{equation*}
For $t\in[-1,0],\frac{w}{4}\geq M^-, k\geq M^-$,
we have
\[\int_{B_1\times\{t\}}|{v_k^+}\zeta|^2 dx+\int_{-1}^t\left\|\nabla\l v_k^+\zeta\r\right\|^2_{2,B_1\times \{s\}}ds \lesssim (C_1^2+C_2) w^{2} A_{k;1}+\]
\begin{equation}\label{inteplus}
r^{\delta_1}w^{\frac{2}{m}}A_{k;q_1}^{q_1}+
C_1 r^{\delta_2}w^{1+\frac{1}{m}} A_{k;q_2}^{q_2}.
\end{equation}
\end{lemma}
\begin{proof}
Let us only prove the second inequality. After multiplying \eqref{eqnu} by $v_k^+\zeta^2$ and doing integration in space as well as from $0$ to $t$, we get
\[w^\alpha m^{-1}\int_{B_1\times\{t\}}\left(\int_0^{v_k^+}(k+\xi)^{-\alpha}\xi d\xi\right)\zeta^2 dx+\int_{-1}^t\left\|\nabla \l v_k^+\zeta\r\right\|^2_{2,B_1\times \{s\}}ds \]
\[\leq 2C_1^2\int_{-1}^t\left\|v_k^+\right\|^2_{2,B_1}ds+2C_2 w^\alpha m^{-1} \int_{-1}^t\int_{B_1}\left(\int_0^{v_k^+}(k+\xi)^{-\alpha}\xi d\xi\right)\zeta dxds+\]
\[r\int_{-1}^{t}\int_{B_1}v^{\frac{1}{m}}\tilde{V}\nabla( v_k^+ \zeta^2) dxds+
2C_1 r\iint_{Q_1}v^{\frac{1}{m}}|\tilde{V}| v_k^+ \zeta dxds.\]
Since $v_k^++k\leq w$, we know
\begin{equation}\label{know}
\frac{1}{2}w^{-\alpha}(v_k^+)^2\leq \int_0^{v_k^+}(k+\xi)^{-\alpha}\xi d\xi\leq \int_0^{v_k^+}(k+\xi)^{\frac{1}{m}} d\xi\leq w^\frac{1}{m}v_k^+.
\end{equation}
The term $r\int_{-1}^t\int_{B_1}v^{\frac{1}{m}}\tilde{V}\nabla (v_k^+ \zeta^2) dxds$ is bounded by
\[2r^2\int_{-1}^t\int_{B_1}v^{\frac{2}{m}}|\tilde{V}|^2\zeta^2\chi_{\{ v>k\}}dxds +\frac{1}{2}\int_{-1}^t\int_{B_1}|\nabla \l v_k^+ \zeta\r|^2dxds+r\int_{-1}^t\int_{B_1}v^{\frac{1}{m}}|\tilde{V}|v_k^+\zeta|\nabla \zeta|dxds.\]
From the above inequality we deduce
\[\int_{B_1\times\{t\}}|{v_k^+}\zeta|^2 dx+\int_{-1}^t\left\|\nabla\l  v_k^+\zeta\r \right\|^2_{2,B_1\times \{s\}}ds \lesssim C_1^2\left\|v_k^+\right\|^2_{2,Q_1}+C_2 w \iint_{Q_1}v_k^+ dxdt+\]
\begin{equation*}
r^2\iint_{Q_1}v^{\frac{2}{m}}|\tilde{V}|^2 \zeta^2\chi_{\{ v>k\}}dxdt+C_1r\iint_{Q_1}v^{\frac{1}{m}}|\tilde{V}| v_k^+ \zeta dxds.\end{equation*}
We denote the last two terms in the above by $X$. Note $v_k^+\lesssim w$, therefore
\[\left\|v_k^+\right\|^2_{2,Q_1}\leq w^2 A_{k;1},\; \left\|v_k^+\right\|_{1,Q_1}\leq w A_{k;1}.\]
Recalling that $A_{k;1}= meas\{Q_1\cap \{v>k\}\}$, it follows that  
\begin{equation}\label{medstep}
\int_{B_1\times\{t\}}|{v_k^+}\zeta|^2(t) dx+\int_{-1}^t\left\|\nabla \l v_k^+\zeta\r\right\|^2_{2,B_1\times \{s\}}ds \lesssim (C_1^2+C_2) w^2 A_{k;1}+X.
\end{equation}

Now we bound the term $X$.
Since $\tilde{V}(x,t)=V(rx,r^2w^{-\alpha}t)$, by the assumption, for each time $t$
\begin{equation}\label{lpV}\|\tilde{V}(\cdot,t)\|_p=r^{-\frac{d}{p}}\|V(\cdot,t)\|_p \lesssim r^{-\frac{d}{p}}.\end{equation}
Then recalling $q_1:=1-\frac{2}{p}$,
\[ r^2\iint_{Q_1} |\tilde{V}|^2\zeta^2\chi_{ \{v>k\}}dxdt\leq r^2\int_{-1}^0 \l\int_{B_1}|\tilde{V}|^p dx\r^{\frac{2}{p}}  \l\int_{B'} \chi_{v>k}dx\r^{q_1} dt\]
\[\leq \int_{-1}^0 r^2\|\tilde{V}\|_{p}^2 \l\int_{B'} \chi_{v>k}dx\r^{q_1} dt\lesssim r^{2-\frac{2d}{p}}A_{k;q_1}^{q_1}.
\]
Similarly, for  $q_2$ satisfying $ \frac{1}{p}+q_2=1$ we have
\[ r\iint_{Q_1} |\tilde{V}|\zeta\chi_{ v>k}dxdt \lesssim r^{1-\frac{d}{p}}A_{k;q_2}^{q_2}.
\]

Combining with \eqref{medstep}, this immediately gives \eqref{inteplus} by the assumptions. Parallel argument applies for the first inequality, except that instead of \eqref{know} we apply
\[\frac{1}{2}k^{-\alpha}(v_k^-)^2\leq \int_0^{v_k^-}(k-\xi)^{-\alpha}\xi d\xi\leq k^\frac{1}{m}v_k^-\]
and the bounds of $v,\tilde{V}$.
\end{proof}

\begin{corollary}\label{lemmaremark}
Under the assumptions of Lemma \ref{inte}. If there exists some universal constants $c,\epsilon>0$ such that
\[k\geq cw,\;r^{\delta_1}w^\frac{2}{m}\leq r^\epsilon w^2,\; r^{\delta_2}w^{1+\frac{1}{m}}\leq r^\epsilon w^2.\] 
Then we have
\begin{equation}\label{clean}\int_{-1}^0\left\|\nabla\l v_k^+\zeta\r\right\|^2_{2,B_1\times \{s\}}ds \leq C |M^+-k|^{2} A_{k;1}+Cr^\epsilon w^2 A_{k;1}^{q_1}.\end{equation}
\end{corollary}
\begin{proof}
The proof follows from a straightforward modification of the one of Lemma \ref{inte}. First, by the assumptions we can replace the second and third inequalities in \eqref{know} by
\[ \int_0^{v_k^+}(k+\xi)^{-\alpha}\xi d\xi \leq k^{-\alpha}\int_0^{v_k^+}\xi d\xi\lesssim w^{-\alpha}(v_k^+)^2\lesssim w^{-\alpha}|M^+-k|^2.
\]
Second by H\"{o}lder's inequality it is not hard to see that $A_{k,q}$ is increasing in $q$ for $q\in (0,1]$ i.e. $A_{k;q_1}\leq A_{k;q_2}\leq A_{k;1}$.
With these two and the previous proof, we conclude with the clean expression \eqref{clean}.
\end{proof}

The first energy inequality in Lemma \ref{inte} will be used in Proposition \ref{prop1}. The second one will be used in Lemma \ref{claim4} and we will apply \eqref{clean} in Lemma \ref{claim3}.

\bigskip

Next we prove two propositions which regards oscillation reduction. The first one implies that under a suitable assumption the solution is bounded away from $0$ with certain amount. The other shows that if the assumption is not satisfied, then the supremum of the solution decreases once we look at a smaller parabolic neighborhood.


\begin{proposition}\label{prop1}
Let $p>d+\frac{4}{d+2}$, $\alpha=\frac{m-1}{m}$ and $\delta_0=\l1-\frac{1}{m}\r/\l1-\frac{d}{p}\r$. Suppose $\nu$ solves \eqref{eqcontv} in a neighbourhood of $Q(r,w^{-\alpha})$ for some $r,w>0$. Denote
$ M^-=\inf\left\{\nu, (x,t)\in Q(r,w^{-\alpha})\right\}$
and let us assume that
\begin{equation}\label{improvement}
w \geq osc_{Q(r,w^{-\alpha})}\nu;\text{ and }\;M^-\leq \frac{w}{4}.
\end{equation}  
Then there exists $c_0\in(0,1)$ that only depends on $m, p$ and $\|V\|_{L^p(Q(r,w^{-\alpha}))}$ such that the following holds: for all $0<r<w^{\delta_0}$ if 
\begin{equation}\label{lemcond}
meas\left\{(x,t)\in Q(r,w^{-\alpha}), \nu(x,t)\geq {M^-}+\frac{w}{2}\right\}\geq (1-c_0)|Q(r,w^{-\alpha})|,
\end{equation}
then 
\[\nu|_{Q(\frac{r}{2},{w}^{-\alpha})}\geq M^- + \frac{w}{4}.\]
\end{proposition}

\begin{proof}

Recall that 
$v(x,t)$ defined in \eqref{alpha:rescale} satisfies \eqref{eqnu} in $Q_1$.
Set
\[r_n:=\frac{1}{2}+2^{-n},\; \tilde{Q}_n=Q({r_n},1),\;k_n:={M^-}+\frac{w}{4}+\frac{w}{2^{n+2}},\]
\[\tilde{B}_{n;q}=\l\int_{-r_n^2}^0\l\int_{B_{r_n}}\chi_{v(x,t)<k_n}dx\r^qdt\r^{\frac{1}{q}}.\]
Pick $\zeta_n\in C_0^\infty(\tilde{Q}_n\cup \left( \tilde{Q}_{n}+(0,2^{-n})\right))$ which equals its maximum $1$ in $\tilde{Q}_{n+1}$. Since $r_n^2-r_{n+1}^2\sim 2^{-n}$, we can assume
\[|\nabla\zeta_n|\lesssim 2^{n}, \quad|\Delta\zeta_n^2|\lesssim 4^{n}, \quad |\partial_t\zeta_n|\lesssim 2^{n}. \]

Recall the notation $v_{k}^-:=\max\left\{k-v,0\right\}$. By Lemma \ref{inte}, after integration we have
\begin{equation*}
\esssup_{-{r_{n+1}}^2\leq t\leq 0}\int_{B_{r_{n+1}}\times\left\{t\right\}}|{v_{k_n}^-}|^2 dx+\int_{-r_{n+1}}^t\left\|\nabla\l v_{k_n}^-\zeta_n\r\right\|^2_{2,B_{r_{n}}\times \left\{s\right\}}ds
\end{equation*} 
\begin{equation*}
\lesssim 4^n w^2\tilde{B}_{n;1}+
r^{\delta_1}w^\frac{2}{m}\tilde{B}_{n;q_1}^{q_1}+
2^n  r^{\delta_2}w^{1+\frac{1}{m}}\tilde{B}_{n;q_2}^{q_2}. \end{equation*}
Unravelling the definition and condition we have $r^{\delta_1}w^\frac{2}{m}\leq w^2, r^{\delta_2}w^{1+\frac{1}{m}}\leq w^2$.
Therefore if taking supremum of $t\in [-r_{n+1}^2,0]$ as well as $t=0$, we obtain 
$$
\left\|{v_{k_n}^-}\zeta_n\right\|^2_{V^{1,0}}:=\esssup_{-{r_{n+1}}^2\leq t\leq 0}\int_{B_{r_{n}}}|{v_{k_n}^-}\zeta_n|^2(\cdot,t) dx+\int_{-r^2_{n+1}}^0\left\|\nabla\l v_{k_n}^-\zeta_n\r\right\|^2_{2,B_{r_{n}}\times \{s\}}ds 
$$
\begin{equation}\label{eqn:V10}\lesssim 4^n w^2\tilde{B}_{n;1}+w^2 \tilde{B}_{n;q_1}^{q_1}+2^n w^2\tilde{B}_{n;q_2}^{q_2},\end{equation}
By Sobolev type embedding (see page 76 in \cite{lady}),
\[\left\| v_{k_n}^-\zeta_n\right\|^2_{L^2(B_{r_{n}}\times [{-r_{n+1}^2},0])} \lesssim \left\|{v_{k_n}^-}\zeta_n\right\|^2_{V^{1,0}}\times \tilde{B}_{n;1}^{\frac{2}{d+2}}.\]
So by \eqref{eqn:V10}, we get
\begin{equation}\label{3.112}
\left\| v_{k_n}^-\zeta_n\right\|^2_{L^2(B_{r_{n}}\times [{-r_{n+1}^2},0])} \lesssim \l4^n \tilde{B}_{n;1}+ \tilde{B}_{n;q_1}^{q_1}+2^n \tilde{B}_{n;q_2}^{q_2}\r w^2 \tilde{B}_{n;1}^{\frac{2}{d+2}}.
\end{equation}

\medskip

By definition, $v^-_{k_{n}}\geq \frac{w}{2^{n+3}}$ in $\tilde{Q}_{n+1}\cap \{v<k_{n+1}\}$. Then
$$
\left\| v_{k_n}^-\zeta_n\right\|^2_{L^2(B_{r_{n}}\times [{-r_{n+1}^2},0])}
\geq \iint_{\tilde{Q}_{n+1} \cap \{v_{k_{n+1}}^->0\}} \l w\;2^{-n-3}\r^2dx  dt \geq w^{2}2^{-2 n-6}\tilde{B}_{n+1;1}.$$
Putting above two computations together, we arrive at 
\begin{equation}\label{inequal}
\tilde{B}_{n+1;1}\lesssim 4^n\l4^n \tilde{B}_{n;1}+ \tilde{B}_{n;q_1}^{q_1}+2^n \tilde{B}_{n;q_2}^{q_2}\r \tilde{B}_{n;1}^{\frac{2}{d+2}}.
\end{equation}

First let us show the result when $p>d+2$. Notice the length of the time interval is bounded by $1$. By definition, $\tilde{B}_{n;q}$ is monotone in $q\in [0,1]$. In particular since $q_1<q_2<1$, 
\[\tilde{B}_{n;q_1}\leq \tilde{B}_{n;q_2}\leq \tilde{B}_{n;1}.\]
Also $\tilde{B}_{n;1}$ is obviously bounded, we have
\begin{equation}\label{iteration..}\tilde{B}_{n+1;1}\leq C16^n \tilde{B}_{n;1}^{q_1+\frac{2}{d+2}}.\end{equation}
Here $C$ is bounded by a universal constant and \[q_1+\frac{2}{d}=1-\frac{2}{p}+\frac{2}{d+2}\] 
is strictly greater than $1$ if $p>d+2$. Then iterating \eqref{iteration..} finishes the proof once the starting point (of the iteration) $\tilde{B}_{0,1}$ is small enough. And this is the same as $c_0$ in \eqref{lemcond} being small enough. For more details, we refer readers to Lemma 2.2 \cite{dibenedetto1}. The choice of $c_0$ only depends on $C,q_1+\frac{2}{d+2}$ in \eqref{inequal} (and of course on $m,V$) which is an universal constant independent of $w$.
\medskip

For $p\in (d+\frac{4}{d+2},d+2]$, we continue with the argument described before \eqref{3.112}. 
For any $\alpha\in(0,1)$, applying Sobolev embedding in space gives 
\[\int_{-r_{n+1}^2}^0 \left\|v_{k_n}^-\right\|^{2\alpha}_{L^2(B_{r_{n+1}})} dt\]
\begin{align*}&\leq \int_{-r_{n+1}^2}^0 \left\|v_{k_n}^-\zeta_n\right\|^{2\alpha}_{L^2({B_{r_{n}}})}  dt\leq C\int_{-r^2_{n+1}}^0 \left\|\nabla \l v_{k_n}^-\zeta_n\r\right\|_{L^\frac{2d}{d+2}(B_{r_{n}})}^{2\alpha} dt\\&
\leq C \int_{-r_{n+1}^2}^0
\left\|\chi_{v^-_{k_n}>0}\right\|^{2\alpha}_{L^{d}({B_{r_{n+1}}})}
\left\|\nabla \l v_{k_n}^-\zeta_n\r\right\|^{2\alpha}_{L^2({B_{r_{n+1}}})}
  dt\\
&\leq C \l \int_{-r_{n+1}^2}^0 \left\|\chi_{v^-_{k_n}>0}\right\|^{\frac{2\alpha}{1-\alpha}}_{L^{d}({B_{r_{n+1}}})}
dt\r^{1-\alpha} \l\int_{-r_{n+1}^2}^0 \left\|\nabla \l v_{k_n}^-\zeta_n\r\right\|^{2}_{L^2({B_{r_{n+1}}})} dt\r^\alpha. \end{align*}
Pick $\alpha=q_1=1-\frac{2}{p}$, then the above is bounded by 
\[
\lesssim \l\tilde{B}_{n;\frac{p-2}{d}}\r^{\frac{2\alpha}{d}}
\l\int_{-r_{n+1}^2}^0 \left\|\nabla \l v_{k_n}^-\zeta_n\r\right\|^{2}_{L^2({B_{r_{n+1}}})} dt\r^\alpha.
\]
By \eqref{eqn:V10} and monotonicity of $\tilde{B}_{n;q}$ in $q$, 
\[\lesssim \l\tilde{B}_{n;1}\r^{\frac{2\alpha}{d}}
\l 4^n w^2\tilde{B}_{n;1}+w^2 \tilde{B}_{n;q_1}^{q_1}+2^n w^2\tilde{B}_{n;q_2}^{q_2}\r^\alpha\]
\[\lesssim \l\tilde{B}_{n;1}\r^{\frac{2\alpha}{d}}
\l 4^n w^2\tilde{B}_{n;1}^{q_2}+w^2 \tilde{B}_{n;q_1}^{q_1}\r^\alpha.\]

On the other hand since $v^-_{k_{n+1}}\geq \frac{w}{2^{n+3}}$ in $\tilde{Q}_{n+1}\cap \{v>k_n\}$, we obtain
$$
w^{2\alpha}2^{-2\alpha n}\tilde{B}_{n+1;q_1}^{q_1} \leq \int_{-r_{n+1}^2}^0\l \int_{B_{r_{n+1}}\cap \{v_{k_{n+1}^-}>0\}} \l w\;2^{-n-3}\r^2dx\r^\alpha  dt
$$
$$
\leq \int_{-r_{n+1}^2}^0 \left\|v_{k_n}^-\right\|^{2\alpha}_{L^2(B_{r_{n+1}})} dt.
 $$

Putting together what we have,
\[w^{2\alpha}2^{-2\alpha n}\tilde{B}_{n+1;1-\frac{2}{d}}^{q_1}
\lesssim \l\tilde{B}_{n;1}\r^{\frac{2\alpha}{d}}\l4^n w^2\tilde{B}^{q_2}_{n;1}+w^2 \tilde{B}_{n;q_1}^{q_1}\r^\alpha\]
which simplifies to
\begin{equation}\label{firstone}\tilde{B}_{n+1;q_1}
\lesssim 16^n \tilde{B}_{n;1}^{\frac{2}{d}} \tilde{B}_{n;1}^{q_2}+4^n\tilde{B}_{n;1}^{\frac{2}{d}} \tilde{B}_{n;q_1}^{q_1}.\end{equation}

Here we have one inequality. By \eqref{inequal}
we proved
\[\tilde{B}_{n+1;1}\lesssim 4^n\l4^n \tilde{B}_{n;1}+ \tilde{B}_{n;q_1}^{q_1}+2^n \tilde{B}_{n;q_2}^{q_2}\r \tilde{B}_{n;1}^{\frac{2}{d+2}}.\]
By monotonicity,
\begin{equation}\label{secondone}\tilde{B}_{n+1;1}\lesssim  16^n \tilde{B}^{q_2+\frac{2}{d+2}}_{n;1}+ 4^n\tilde{B}_{n;q_1}^{q_1} \tilde{B}_{n;1}^{\frac{2}{d+2}}.\end{equation}
Let $a_n=\tilde{B}_{n;1},b_n=\tilde{B}_{n;q_1}$.
Then the proof is finished by applying Lemma \ref{twoseq} and using both \eqref{firstone} and \eqref{secondone}.
\end{proof}

We state the Lemma~\ref{twoseq}, whose proof will be given in the appendix.

\begin{lemma}\label{twoseq}
Suppose we have two sequences $\{a_n\},\{b_n\}$ such that 
\[1>a_0\geq b_0 \geq a_1\geq b_1 ...\geq a_n\geq b_n\geq 0,\]
and there exists a constant $C_1$ that
\begin{align*} &b_{n+1}\leq C_1^n a_{n}^{q_2+\frac{2}{d}}+C_1^n a_n^\frac{2}{d}b_n^{q_1};\\
&a_{n+1}\leq C_1^n a_n^{q_2+\frac{2}{d+2}}+C_1^nb_n^{q_1}a_n^\frac{2}{d+2}.\end{align*}
Then if $p>d+\frac{4}{d+2}$ and $a_0$ is small enough, 
\[\lim_{n\to\infty}a_n=\lim_{n\to\infty}b_n=0.\]
\end{lemma}


Now we proceed to the second proposition. 

\begin{proposition}\label{prop2}
Let $\nu$ and $c_0$ be as in Proposition \ref{prop1}.
Suppose \eqref{improvement} holds while \eqref{lemcond} is not satisfied. Then there exists universal constants $c_1,c_2$ which only depends on $c_0, p$ and $\|V\|_{L^p(Q_1)}$such that the following is true: 

For $r$ satisfying $r< c_1w^{c_2}$, there exists some $\eta\in(0,1)$ such that
\[\nu|_{Q\l\frac{r}{2},\frac{c_0}{2} w^{-\alpha}\r}\leq \eta w.\]
The constant $\eta$ is independent of $w$ which may depend on $c_0$.
\end{proposition}

The proof of the proposition rests on a number of lemmas which are variants of Lemma 6.1, 6.2, 6.3, 6.4 in \cite{dibenedetto1}. We will sketch the proof for some lemmas and emphasize on the differences.

\medskip

Lemma \ref{claim1}--Lemma \ref{claim4} stated below are proven under the conditions of Proposition \ref{prop2} and, for $c_0$ given in Proposition~\ref{prop1} we have 
\begin{equation}\label{extracond}M^+-M^-\geq \frac{w}{2}+c_0w,
\end{equation}
where $M^+$ and $M^-$ denote respectively the supremum and infimum of $\nu$  in $Q (r,w^{-\alpha})$. 
\medskip
Let $v(x,t)$ be as given in \eqref{alpha:rescale} and define

\begin{equation}\label{notation}A_{k,R}(t):=\left\{x\in B_R: v(x,t)>k\right\}.\end{equation}
We denote $A_k(t) := A_{k,1}(t)$.  

\begin{lemma}\label{claim1}Let $k_1=M^+-c_0w$. There exists $\tau\in (-1,-\frac{c_0}{2})$ such that
\[|A_{k_1}(\tau)|\leq (1-{c_0})\l 1-\frac{c_0}{2}\r^{-1}|B_1|.\]
\end{lemma}
\begin{proof}
Observe that, by \eqref{extracond}, $k_1\geq M^-+\frac{w}{2}$. If the claim is false, 
\begin{align*}meas\left\{(x,t): |x|\leq 1, t\in (-1,-\frac{c_0}{2}),  v>M^-+\frac{w}{2}\right\}\geq \int_{-1}^{-\frac{c_0}{2}} |A_{k_1}(t)|dt
>(1-{c_0})|B_1|
\end{align*}
which agrees with \eqref{lemcond} and thus contradicts with the condition of Proposition~\ref{prop2}.
\end{proof}

\begin{lemma}\label{claim2}  Let $c_0$ as given in Proposition~\ref{prop1}, $M^+$ in \eqref{extracond} and $A_k(t)$ in \eqref{notation}. There exist universal constants $c_1,c_2>0$ and a sufficiently large positive integer $q=q(c_0)$ which is independent of $w$ such that if $r<c_1w^{c_2}$ then for $k_2 = M^+ - \frac{c_0}{2^q}w$ we have 
\[|A_{k_2}(t)|\leq \l 1-\frac{c_0^2}{4}\r|B_1| \hbox{ for } t\in [-\frac{c_0}{2},0].
\]
\end{lemma}
\begin{proof}

Without loss of generality we may assume that $c_0<1$. We follow the outline of the proof  for Lemma 6.2  in \cite{dibenedetto1}. The additional ingredient is that we need to consider the effect of the drift term and give a clear description of how small $r$ need to be. For $q>3$, consider
\[\psi(x)=\log^+\left(\frac{c_0 w}{c_0w-(x-(M^+-c_0w))^++\frac{c_0 w}{2^q}}\right).\]
Then 
\begin{equation}\label{bounds00}
0\leq \psi(x)\leq q\log 2,\hbox{ and }  \psi'(x)\in\left[0,\frac{2^q}{c_0w}\right]\ \hbox{ for }x\in [0,M^+].
\end{equation}
Let $\zeta$ be a cutoff function in $B_1$ that
\[\zeta=1 \text{ in }B_{1-\lambda },\; \zeta\in [0,1],\;  |\nabla\zeta|\leq \frac{2}{\lambda}\]
where $\lambda\in (0,1)$ is to be determined.

\medskip

 Consider $\phi=(\psi^2)'(v)\zeta^2(x)$. Let $\tau$ be from Lemma \ref{claim1} and set $Q^\tau:=B_1\times [\tau,t]$. By calculating $\iint_{Q^\tau}mv^\alpha v_t (\psi^2)'\zeta^2dxdt$ and using equation \eqref{eqnu}, we find
\[w^\alpha\iint_{Q^\tau}\partial_t\phi dxdt=-\iint_{Q^\tau} m\left(\nabla  v+r v^\frac{1}{m}{\tilde{V}}\right)\left(\nabla ( v^\alpha(\psi^2)'\zeta^2)\right)dxdt.\]
Notice that $(\psi^2(t))''=2(1+\psi(t))(\psi'(t))^2$. So
\[w^\alpha\int_{B_1\times\{t\}}\psi^2( v)\zeta^2dx-w^\alpha\int_{B_1\times\{\tau\}}\psi^2( v)\zeta^2 dx\]
\[+\underbrace{\iint_{Q^\tau} v^{-\frac{1}{m}}(\psi^2)'|\nabla  v|^2\zeta^2dxdt+\iint_{Q^\tau}  v^\alpha(1+\psi)|\psi'|^2|\nabla  v|^2\zeta}_{X_1:=}\]
\[\lesssim_m \underbrace{\iint_{Q^\tau}|\nabla  v\; \psi'\psi^\frac{1}{2} v^\frac{\alpha}{2}| v^\frac{\alpha}{2}\psi^\frac{1}{2}\zeta\lambda^{-1}}_{X_2:=}+\underbrace{r\iint_{Q^\tau}|\nabla( v^\alpha(\psi^2)'\zeta^2)|\cdot|{\tilde{V}} v^\frac{1}{m}|dxdt}_{X_3:=}.\]
Since $ v\leq M^+\sim w$, $\psi\lesssim q$,
\[X_2\leq C w^\alpha q\lambda^{-2}+o(1) X_1.\]
From the H\"{o}lder inequality and the fact that $|\nabla\zeta|\lesssim \lambda^{-1},\zeta\in [0,1]$,
\begin{align*}
X_3 &\leq  Cr\iint_{Q^\tau}\left(|\nabla v\; \psi^\frac{1}{2}(\psi')^\frac{1}{2}\zeta|\psi^\frac{1}{2}(\psi')^\frac{1}{2}|{\tilde{V}}|\zeta+ v(1+\psi)|\psi'|^2|\nabla v||{\tilde{V}}|\zeta^2+\lambda^{-1} v\psi\psi'|{\tilde{V}}|\zeta\right)dxdt \\ 
&\lesssim  o(1)X_1+r^2\iint_{Q^\tau}\l v^\frac{1}{m}\psi\psi'+ v^{2-\alpha}(1+\psi)|\psi'|^2\r|{\tilde{V}}|^2dxdt+r\iint_{Q^\tau}\lambda^{-1} v\psi\psi'|{\tilde{V}}|dxdt.
\end{align*}
Recall \eqref{bounds00} and that $ v\lesssim w\lesssim 1$. Hence we obtain 
\[X_3\lesssim  o(1)X_1+({4^q qr^2}/{c^2_0 w^\alpha})\iint_{Q^\tau}|{\tilde{V}}|^2dxdt+({2^q q r}/{c_0 \lambda}) \iint_{Q^\tau}|{\tilde{V}}|dxdt \]
Now by \eqref{lpV}
\[X_3\lesssim  o(1)X_1+({4^q qr^{\delta_1}}/{c^2_0 w^\alpha})+({2^q q r^{\delta_2}}/{c_0 \lambda}) .
\]

Let $A_{k,R}(t)$ be as given in \eqref{notation}. Computations in the proof of Lemma 6.2 \cite{dibenedetto1} yield
\begin{align*}\int_{B_1\times\{t\}}\psi^2( v)\zeta^2dx &\geq ((q-1)\log 2)^2|A_{k_2,1-\lambda}(t)|,\\ 
\int_{B_1\times\{\tau\}}\psi^2( v)dx & \leq (q\log 2)^2|A_{k_1}(\tau)|\end{align*}
where $k_1$ is as defined in Lemma \ref{claim1}. From the above,
\[
|A_{k_2,1-\lambda}(t)|\leq \l\frac{q}{q-1}\r^2 |A_{k_1}(\tau)|+\frac{Cq}{\lambda^2(q-1)^2}+
\frac{C4^q  r^{\delta_1}q}{c^2_0 (q-1)^2 w^{2\alpha}}+\frac{C2^q  r^{\delta_2}q}{c_0 \lambda (q-1)^2w^\alpha  }.
\]
And we have
\[|A_{k_2,1-\lambda}(t)|\geq |A_{k_2}(t)|-|B_1\backslash B_{1-\lambda}|\geq  |A_{k_2}(t)|-Cd\lambda|B_1|.\]
By Lemma \ref{claim1} and $q\geq 3$, we obtain
\begin{equation}\label{rhs}
|A_{k_2}(t)|\leq 
\left(\l\frac{q}{q-1}\r^2\frac{1-c_0}{1-\frac{1}{2}c_0}+C_0d\lambda\right)|B_1|+C_1\l\frac{1}{\lambda^2q}+
\frac{4^q  r^{\delta_1}}{c^2_0 q w^{2\alpha}}+\frac{2^q  r^{\delta_2}}{c_0 q\lambda w^\alpha  }\r,\end{equation}
where $C_0$ and $C_1$ are universal constants.

Let us now choose $\lambda$ and $q$ such that 
\[\lambda:= \frac{1}{4C_0d}c_0^2, \quad  \l\frac{q}{q-1}\r^2\leq \l 1-\frac{c_0}{2}\r\l 1+c_0\r,\; \frac{C_1}{\lambda q}\leq \frac{1}{4}c_0^2|B_1|.\] 
It is possible to choose such $q$ since for $c_0$ small, $(1-\frac{c_0}{2})(1+c_0)>1$. Due to the drift term we require \[({4^q r^{\delta_1}}/{c^2_0 w^{2\alpha}}+{2^q r^{\delta_2}}/{c_0 \lambda w^\alpha}) \lesssim \frac{1}{4}q c_0^2|B_1|.\] Since $c_0$ is fixed and $\lambda(c_0),q(c_0)$ are fixed, this condition is equivalent to $r\leq c_1w^{c_2}$ for some fixed $c_1(c_0),c_2(c_0)>0$.

\medskip

Finally we can conclude with the right hand side of \eqref{rhs} $\leq (1-(\frac{c_0}{2})^2)|B_1|$.

\end{proof}

\begin{lemma}\label{claim3} Let $q$  be as given in Lemma~\ref{claim2}. Then for any $\gamma\in (0,1)$ there exists $c(\gamma,c_0,q)>0$ and $p_0(\gamma,c_0,q)>q$  such that the following holds:  if $r$ satisfies the assumption given in Lemma~\ref{claim2} and further satisfies $r \leq c$, then 
\[\left|\left\{ (x,t)\in Q\l 1,\frac{c_0}{2} \r,  v>M^+-\frac{c_0}{2^{p_0}}w\right\}\right|\leq \gamma\left|Q(1,\frac{c_0}{2} )\right|. \]
\end{lemma}

\begin{proof}
The lemma is a variant of Remark 6.1, Lemma 6.3, 6.4 \cite{dibenedetto1}. 

\medskip

Write $\left\{ u>k\right\}:=\left\{x\in B_1,\; u>k\right\}$ for any $u\in W^{1,2}(B_1)$.
Lemma 6.3 \cite{dibenedetto1} says that for any $l>k$,
\begin{equation}\label{101}
(l-k)|\left\{ u>k\right\}|^{1-\frac{1}{d}}\leq \frac{C}{|B_1|-|\left\{ u>k\right\}|}\int_{\left\{ u>k\right\}\backslash \left\{ u>l\right\}}|\nabla u|dx.
\end{equation}
We will chosider $l=k_{s+1}, k=k_s$ in above inequaltiy, wher ${k_s}:=M^+-\frac{c_0}{2^s}w$, where $s$ is a sufficiently large integer to be determined below.

\medskip

Let us choose $\lambda = \lambda(\gamma, c_0)$ such that 
\begin{equation}\label{lambdadiv2}
\left|Q\l 1,\frac{c_0}{2} \r\big\backslash Q\l\lambda,\frac{c_0}{2}\r\right|\leq \frac{\gamma}{2}\left|Q\l 1,\frac{c_0}{2}\r\right|.
\end{equation} 
With above choice of $\lambda$, let $0\leq\zeta(x,t)\leq 1$ be a cut-off function compactly supported in $Q\l 1,\frac{c_0}{2} \r$ which equals $1$ in $Q(\lambda,\frac{c_0}{2})$. 

Write 
\[A^\zeta_{k}(t):=\left\{x\in B_1,  v \zeta> k\right\},\;A^\zeta_{k,c_0}:=\left\{(x,t)\in Q(1,\frac{c_0}{2}),  v \zeta> k\right\},\]
\[A_{k,\lambda, c_0}:=\left\{(x,t)\in Q\l \lambda,\frac{c_0}{2} \r,  v > k\right\}.\]
Then
\begin{equation}\label{6.15}
\frac{wc_0}{2^{s+1}}|A^\zeta_{{k_{s+1}}}(t)|^{1-\frac{1}{d}}\leq \frac{C}{meas\left\{B_1\backslash A^\zeta_{{k_s}}(t)\right\}}\int_{A^\zeta_{{k_s}}(t)\backslash A^\zeta_{{k_{s+1}}}(t)}|\nabla ( v\zeta)|dx.\end{equation}
Recall that from \eqref{notation} $ A_{k,R}(t)=\left\{x\in B_R,  v> k\right\}$ and $A_k(t)=A_{k,1}(t)$, so by definitions of the sets 
\[A_{k,\lambda}(t)\subseteq A^\zeta_{k}(t)\subseteq A_{k}(t).\]
By Lemma \ref{claim2}, for any $t\in[-\frac{c_0}{2},0]$,
\[meas\left\{B_1\backslash A^\zeta_{{k_s}}(t)\right\}\geq meas\left\{B_1\backslash A_{{k_s}}(t)\right\}\geq \l\frac{c_0}{2}\r^2|B_1|.\]
Since $|A^\zeta_{{k_s}}(t)|$ is bounded by $|B_1|$, 
\[C|A^\zeta_{{k_{s+1}}}(t)|^{1-\frac{1}{d}}\geq |A^\zeta_{{k_{s+1}}}(t)|\geq |A_{{k_{s+1}},\lambda}(t)|.\]
Then
\[ |A_{{k_{s+1}},\lambda, c_0}|\leq C\int_{-\frac{c_0\lambda}{2}}^0 |A^\zeta_{{k_{s+1}}}(t)|^{1-\frac{1}{d}}dt.\]
After integrating \eqref{6.15}, H\"{o}lder inequality yields that 
\[\frac{wc_0}{2^{s+1}}|A_{{k_{s+1}},\lambda,c_0}|\leq \frac{C}{c_0^2}\int^0_{-\frac{c_0}{2}}\int_{A_{{k_s}}^\zeta(t)\backslash A_{{k_{s+1}}}^\zeta(t)}|\nabla ( v\zeta)|dxdt \]
\[\leq\frac{C}{c_0^2}\left(\iint_{A_{{k_s},c_0}^\zeta\backslash A_{{k_{s+1},c_0}}^\zeta}|\nabla ( v\zeta)|^2dxdt\right)^\frac{1}{2}\left|A_{{k_s},c_0}^\zeta\backslash A_{{k_{s+1},c_0}}^\zeta\right|^\frac{1}{2}.
 \]
  
 
Next
according to \eqref{clean} 
\[\iint_{\zeta v>{{k_s}},(x,t)\in Q(1,\frac{c_0}{2})}|\nabla\l  v \zeta\r |^2 dxdt \lesssim_\lambda \frac{c_0^2w^2}{2^{2s}} +r^\epsilon w^2.\]
Then
\[\frac{wc_0}{2^{s+1}}|A_{{k_{s+1}},\lambda}|\leq \frac{C}{c_0^2}\left(  
\frac{c_0^2w^2}{2^{2s}} +r^\epsilon w^2
\right)^\frac{1}{2}\left|A_{{k_s}}^\zeta\backslash A_{{k_{s+1}}}^\zeta\right|^\frac{1}{2}.\]
Now we let $r$ be small enough that $r^\epsilon 4^{p_0}c_0^{-2}\leq 1$.
Then for all $q\leq s\leq p_0-1$ 
\[
|A_{{k_{s+1}},\lambda,c_0}|^2\leq \frac{C_0}{c_0^4} |A_{{k_s},c_0}^\zeta\backslash A_{{k_{s+1},c_0}}^\zeta|. \]

As in \cite{dibenedetto1}, since the sum of $|A_{{k_s},c_0}^\zeta\backslash A_{{k_{s+1},c_0}}^\zeta|$ is uniformly bounded by $|B_1|$. If $p_0$ is large enough, there is $s_0\in [q,p_0-1]$ that
\[C_0\left|A_{k_{s_0},c_0}^\zeta\backslash A_{k_{s_0+1},c_0}^\zeta\right|\leq \frac{C c_0^4 }{p_0-q-1}\leq c_0^4\l\frac{c'\gamma }{2}\r^\frac{1}{2}\]
with $c'=|Q\l 1,\frac{c_0}{2} \r|$. Let us choose $s=s_0$. Then 
\[|A_{{k_{s_0+1}},\lambda,c_0}|\leq  \frac{c'\gamma}{2}=\frac{\gamma}{2}\left|Q(1,\frac{c_0}{2})\right|.\]
Consequently 
\[\left|\left\{ v>M^+-\frac{c_0}{2^{p_0}}w \text{ in }Q\l \lambda,\frac{c_0}{2} \r\right\}\right|=\left|A_{k_{p_0},\lambda,c_0}\right|\leq |A_{{k_{s_0+1}},\lambda,c_0}|\leq \frac{\gamma}{2}\left|Q\l 1,\frac{c_0}{2} \r\right| .\]
Note that $p_0$ can be determined by $c_0,q,\gamma$ and so we only need $r\leq c(c_0,q,\gamma)$.

\medskip

Finally from \eqref{lambdadiv2}
\[\left|\left\{ v>M^+-\frac{c_0}{2^{p_0}}w \text{ in }Q\l 1,\frac{c_0}{2} \r\right\}\right|\leq \left|\left\{ v >M^+-\frac{c_0}{2^{p_0}}w \text{ in }Q\l \lambda,\frac{c_0}{2} \r\right\}\right|\]\[+\left|Q\l 1,\frac{c_0}{2} \r\backslash Q(\lambda,\frac{c_0}{2})\right|\leq {\gamma}\left|Q\l 1,\frac{c_0}{2} \r\right|.
\]

\end{proof}

The following lemma helps finding the value of $\gamma(c_0,p_0)$. The proof is parallel to Proposition~\ref{prop1}.

\begin{lemma}\label{claim4} Let $p_0$ be as given in Lemma~\ref{claim3}. Suppose $p>d+\frac{4}{d+2}$. There exists $\gamma\in(0,1)$ independent of $w,r,p_0$ such that if $r<c_1w^c_2$ for some $c_1,c_2$ depending on $c_0,p_0$ and
\[\left|\left\{ (x,t)\in Q(1,\frac{c_0}{2}),  v>\l M^+-\frac{c_0}{2^{p_0}}w\r\right\}\right|\leq \gamma\left|Q(1,\frac{c_0}{2})\right|, \]
then
\[\left|\left\{ (x,t)\in Q\l\frac{1}{2},\frac{c_0}{2}\r,  v>\l M^+-\frac{c_0}{2^{p_0+1}}w\r\right\}\right|=0. \]
\end{lemma}


\medskip

{\bf Proof of Proposition \ref{prop2}}

Without loss of generality, we may assume $c_0<\frac{1}{8}$.
If $M^+-M^-\leq \frac{w}{2}+c_0w$, then $ v$ is bounded by $M^+\leq (\frac{3}{4}+c_0)w$. In this case, taking $\eta\leq (\frac{3}{4}+c_0)$ finishes the proof of the Proposition with, for instance, $c_1=c_2=1$.

\medskip

Otherwise condition \eqref{extracond} is satisfied. In this case we fix $c_1,c_2$ as given in Lemma~\ref{claim1}, $p_0$ as in Lemma~\ref{claim3} and $\gamma$  as in  Lemma \ref{claim4}. By Lemma \ref{claim1} and Lemma \ref{claim2}, we know that the conclusion of Lemma \ref{claim3} is valid for the range of $r$ satisfying $r< c_1 w^{c_2}$.   By Lemma \ref{claim3}, we know that the condition in Lemma \ref{claim4} is satisfied for the specific choice of $p_0$. By Lemma \ref{claim4}  we proved that if \eqref{lemcond} is not satisfied, the solution goes down from above (from $M^+$ to $M^+-c_0 2^{-p_0-1} w$) if restricted to the smaller box $Q(1/2,c_0/2)$. This yields the conclusion with $\eta=1-c_0 2^{-p_0-1}$.

\hfill$\Box$

{\bf Proof of Theorem \ref{thmcont}}
The proof follows an iteration process which was described in the proof of Theorem 7.17 \cite{vazquez}, based on Propositions ~\ref{prop1} and ~\ref{prop2}.

\medskip

Recall that $M:= \sup_{Q_1} u$ and $\alpha=\frac{m-1}{m}$. Fix $(x_0,t_0)\in Q_{\frac{1}{2}}$, without loss of generality we can assume it is $(0,0)$, and let $\nu:= u^{1/m}$.

The goal for the argument below is to obtain
\begin{equation}\label{holder}
\eta^kw \geq {\rm osc}_{Q(a^kr, b^{2k})} \nu \hbox{ for all integers } k,
\end{equation}
where $ a,b,\eta\in (0,1)$ only depends on $M, m, p$, $\|V\|_{L^p(Q_1)}$ and the dimension $d$.

\medskip

We start with some $Q(r,w^{-\alpha})$ for some $w>0, 0<r\leq \frac{1}{2}$ such
 that 
\begin{equation}\label{improvement1}
Q(r,w^{-\alpha})\subset Q_\frac{1}{2},\;w \geq osc_{Q(r,w^{-\alpha})}\nu.
\end{equation}
For example we can take $w=M$.

 Let us start with a given pair of $(r_0,w_0)$ that satisfies \eqref{improvement1}. Below we will generate a sequence of pairs $(r_n, w_n)$ that satisfies \eqref{improvement1}. For each $n$ and the given pair $(r_n, w_n)$ let us denote
$$
M^-_n:=\inf_{Q(r_n,w_n^{-\alpha})}\nu, \quad M^+_n:=\sup_{Q(r_n,w_n^{-\alpha})} \nu.
$$

\medskip

Let $c_1$ and $c_2$ be as given in Proposition~\ref{prop2}. For each given pair $(r_n, w_n)$ the next pair $(r_{n+1}, w_{n+1})$ is generated depending on the following cases.

\begin{itemize}
\item[Case 1:] if $r_n> c_1w_n^{c_2}$, the situation is in some sense better since the oscillation is under control. In order to apply the preceding scheme, let $w_{n+1}=w_n, r_{n+1}=\frac{1}{2}r_n$, and we repeat until it falls into Case 2 or 3.

\item[Case 2:] if $r_n\leq c_1 w_n^{c_2}$ and either $M^-_n \geq \frac{w_n}{4}$ or  \eqref{lemcond} holds, we claim $\nu\in [w_n/4, M_n^+]$ in $Q(\frac{3r_n}{4},w_n^{-\alpha})$. This is trivial if $M_n^-\geq \frac{w_n}{4}$, otherwise with the help of \eqref{lemcond} we can apply Proposition~\ref{prop1}.  Then from classical regularity theory for parabolic equations, it follows that \eqref{holder} holds for $k\geq n$. 

\medskip

\item[Case 3:] We are left with the case $r_n\leq c_1w_n^{c_2}$, $M^-_n < \frac{w_n}{4}$ and \eqref{lemcond} fails. In this case
 Proposition~\ref{prop2} yields constants $0<c_0,\eta<1$ which are independent of $w$ such that
\begin{equation}\label{etaiteration}osc_{Q(\frac{r_n}{2},\frac{c_0}{2} w_n^{-\alpha})}\nu\leq \eta w_n.
\end{equation}
We choose \[w_{n+1}:=\eta w_n, \quad r_{n+1}:=c_3 r_n.\]
Here $c_3^2:=\frac{1}{8}\eta^{{\alpha}}c_0$ is chosen such that $Q(r_{n+1},w_{n+1}^{-\alpha})\subset Q(\frac{r_n}{2},\frac{c_0}{2}w_n^{-\alpha})$. 
From this choice of $c_3$ and \eqref{etaiteration} it follows that  \eqref{improvement1} holds for $(r_{n+1}, w_{n+1})$.

Suppose Case 3 is iterated for $n$ times. Then inside $\{|x|<c_3^nr, t\in(-c_5 c_4^{2n}r^2,0)\}$, the oscillation of $\nu$ is bounded by $\eta^n w$ and here $c_4=\frac{1}{4}c_0,c_5=c_0 w^{-\alpha}$. This yields \eqref{holder} for $k=n$.

\end{itemize}

\hfill$\Box$


\section{Loss of regularity: Examples}

In this section we show by examples that the regularity results obtained in section 3 and 4 are false for drifts in $L^d(\R^d)$. We will discuss examples with both potential vector fields and divergence-free vector fields. 

\subsection{Loss of uniform bound and continuity for potential vector fields}

First let us recall the description of stationary solutions for \eqref{main} with potential vector fields.

\begin{theorem}\label{asymtotic}[\cite{density}, \cite{kimlei}] 
For a radially symmetric, increasing potential $\Phi\in C^\infty(\mathbb{R}^d)$, the following is true:
\begin{itemize}
\item[1.]   The unique stationary solution of \eqref{main}, with a prescribed mass $M$, is of the form
\[\rho_M=\l C(M)-\frac{m-1}{m}\Phi\r_+^\frac{1}{m-1}.\]
\item[2.] Let $\rho$ solve \eqref{main} with $V=\nabla\Phi$ and with smooth compactly supported initial data $\rho_0$ with $\int \rho_0=M$. Then the support of $\rho$ stays bounded for all times, and $\|\rho(\cdot,t)-\rho_M(\cdot)\|_{L^{\infty}(\R^d)}\rightarrow 0$ as $t\rightarrow \infty$.
\end{itemize}
\end{theorem}

Based on above theorem, we give the example showing the loss of uniform boundedness of solutions. 

\begin{theorem}\label{thm5.2}
Let $d\geq 2$ and $1\leq p\leq d$. Then there exists a sequence of vector fields $\{\nabla \Phi_A(x)\}_{A\in\mathbb{N}}$ which are uniformly bounded in $L^p(\mathbb{R}^d)$ such that the following holds.  Let $u_A$ solve \eqref{main} with $V=\nabla\Phi_A$ and with  a smooth, compactly supported initial data $u_0$. Then $\sup_{x\in\R^d, t>0} u^A(x,t)  \to \infty $ as $A\to\infty$.


\end{theorem} 

\begin{proof}

Let $f(r):=\ln\ln\frac{1}{r}$ for $r\in(0,1)$, which satisfies $f'(r) = (1/r)\ln r$ . For each $A>>1$,
consider a radially symmetric and increasing function $\Phi_A(x):= \phi_A(|x|)$ such that 
\begin{itemize}
\item[1.] $\phi_A(r)=-f(r)\quad\text{ if }1/A\leq r \leq (\ln\ln A)^{-1}$;\\
\item[2.] $\phi'_A(r)\leq  -2f'(r)\quad\text{ if }{1}/{(2A)}\leq r \leq 2(\ln\ln A)^{-1}$; \\
\item[3.] $0\leq \phi'_A(r) \leq \min\{1,r^{-d-1}\}$\quad if $ r\leq {1}/{(2A)}$ or $r\geq 2(\ln\ln A)^{-1}$.
\end{itemize}
From Theorem \ref{asymtotic} there is a stationary solution $\rho_1^A$ with total mass $1$, of the  form \[\rho_1^A(x)=\l C_A -\frac{m-1}{m}\Phi_A(x)\r^{\frac{1}{m-1}}_+.\] 

We claim that $\rho_A(0)\to\infty$ as $A\to\infty$. Indeed, otherwise $\sup\rho_A = \rho_A(0)$ are uniformly bounded, and having to reach total mass $1$ the support of $\rho_A$ must stay away from vanishing. Thus  it follows that $\rho_1^A((\ln\ln(A))^{-1}) > 0$ for sufficiently large $A$, and we have 
 \[\rho_1^A(0)=\l C_A -\frac{m-1}{m}\Phi_A(0)\r^{\frac{1}{m-1}}\geq \l\frac{m-1}{m}\Phi_A((\ln\ln A)^{-1})-\frac{m-1}{m}\Phi_A(0)\r^\frac{1}{m-1}.\]
Since $\Phi_A((\ln\ln A)^{-1}) -\Phi_A(0) >f\l\frac{1}{A}\r-f((\ln\ln A)^{-1})\to +\infty$ as $A\to +\infty$, we reached a contradiction. 

Let $u^A$ be a solution to \eqref{main} with initial data
$u_0(x)$. Then by Theorem \ref{asymtotic}, $u^A\to \rho^A_1$ in $L^\infty(\mathbb{R}^d)$ as $t\to \infty$. From above argument we have
$\sup_{t>0} u^A(0,t)=\rho_1^A(0)\to \infty$ as $A\to \infty$.



To finish the proof we only need to check that $\nabla\Phi_A$ is bounded in $L^p(\mathbb{R}^d)$. In fact one can check that the vector fields $\nabla\Phi_A$ is uniformly bounded in $L^{d}_{\log^{q}}$ for  all $0<q< d-1$, where 
\[\left\|V\right\|^d_{L^{d}_{\log^q}}:=\int_{\mathbb{R}^d}|V|^d\max\{\log^q |V|,1\}dx.\]

It is enough to check the region $\frac{1}{A}\leq|x|\leq (\ln\ln A)^{-1}$, since elsewhere property 3. in the construction of $\Phi_A$ guarantees the uniform bound. We have

\begin{align*}
&\int_\frac{1}{A}^{(\ln\ln A)^{-1}} |\nabla \Phi_A|^d \log^{q}|\nabla \Phi_A| dx\lesssim \int_\frac{1}{A}^{(\ln\ln A)^{-1}} |\ln|x||^{-d}|x|^{-d}\log^q |x|^{-1}dx
\\
&=\int_\frac{1}{A}^{(\ln\ln A)^{-1}} 
|\ln r|^{-d+q} r^{-1}d r = \epsilon\l|\ln \ln\ln A|^{-\epsilon}-|\ln A |^{-\epsilon}\r,
\end{align*}
which is uniformly bounded as $A\to\infty$ if $d-1-q = \e>0$. 
\end{proof}


By choosing another family of potentials, we can also check that $L^d(\R^d)$ bound on drifts does not guarantee any modulus of continuity for solutions of \eqref{main} even when the solutions are uniformly bounded.

\begin{theorem}\label{thm:lackofcont}
There exists a family of potentials $\Phi_A$ such that $\nabla\Phi_A\in L^d(\R^d)$ and a family of initial data $u_0^A$ which are uniformly bounded in $L^1(\mathbb{R}^d)\cap L^\infty(\mathbb{R}^d)\cap C^\infty(\mathbb{R}^d)$ such that the following holds: The solutions $u^A$ of \eqref{main} with $V=\nabla\Phi_A$ with initial datas $u_0^A$  stays uniformly bounded but lacks any uniform modulus of continuity as $A\to\infty$. 
\end{theorem}

\begin{proof}
Let $\phi(x)=|x|^2$, and let $\rho$ be a stationary solution of \eqref{main} $\rho$ given in Theorem \ref{asymtotic}, with a sufficiently small mass such that $\rho$ is supported inside of the unit ball.
Let $\phi_A(x):=\phi(Ax)$, and $\rho^A(x):=\rho(Ax)$, which is a stationary solution for $\phi_A$. Let us next modify $\phi_A$ so that $\nabla\phi_A$ is uniformly bounded in $L^d(\R^d)$, let $\Phi_A$ satisfy
\begin{itemize}
\item[1.] $\Phi_A=\phi_A$ if $|x|\leq {1}/{A}$.\\
\item[2.] $|\nabla \Phi_A|\leq |\nabla \phi_A|$ if $|x|\leq {2}/{A}$.\\
\item[3.] $|\nabla \Phi_A|\leq \min\{1,|x|^{-1}\}$ if $|x|\geq {2}/{A}$.\\
\item[4.] $\Phi_A$ is smooth, radially symmetric and increasing.
\end{itemize}
Then $\nabla\Phi_A$ is uniformly bounded in $L^d(\R^d)$ and $\rho_A$ is still a stationary solution for the modified potential $\Phi_A$.

\medskip

 For $A>1$, consider a sequence of functions $u_0^A\geq 0$ such that they are uniformly bounded in $L^1(\mathbb{R}^d)\cap L^\infty(\mathbb{R}^d)\cap C^\infty(\mathbb{R}^d)$ and $\int u_0^A dx= \int \rho^A dx=CA^{-d}$. By Theorem \ref{asymtotic}, the solution $u^A$ of \eqref{main} with initial data $u_0^A$ and with $V=\nabla\Phi_A$ converges uniformly to  $\rho^A=\rho(Ax)$ and $\rho^A$ converges pointwise to a discontinuous function $\rho^{\infty}$ which is $1$ at $x=0$ and zero for sufficiently small $|x|$. It follows that $u^A$ cannot share any uniform modulus of continuity.

We are left to show that $u^A$ is bounded.
To see this let $v^A(x,t):=u^A\l A^{-1} x,A^{-2} t\r$. Then 
\[v^A_t=\Delta \l v^A\r^m+\nabla\cdot \l v^A A^{-1}\l\nabla \Phi_A\r\l A^{-1} x \r\r,\]
\[\text{ and }\left\|A^{-1}\l\nabla \Phi_A\r(A^{-1}x)\right\|^{d+1}_{d+1}=A^{-1}\left\|\nabla \Phi_A\right\|^{d+1}_{d+1}\]
\[\lesssim 
A^{-1}\l\int_0^\frac{2}{A}|A^2 x|^{d+1}dx+
\int_1^\infty|x|^{-d-1}dx+1\r<\infty.\]
The vector field $ A^{-1}\l\nabla \Phi_A\r\l A^{-1} x \r$ are uniformly bounded in $L^{d+1}$ and $v^A(0)$ are uniformly bounded in $L^1(\mathbb{R}^d)\cap L^\infty(\mathbb{R}^d)$. By previous Theorem \ref{uniformb}, $v^A$ are uniformly bounded and so are $u^A$.

\end{proof}

\subsection{Loss of H\"{o}lder regularity for Divergence free vector fields}

In previous subsection we have seen that drifts bounded in $L^d(\R^d)$ and initial data that are bounded in $L^1(\R^d)\cap L^{\infty}(\R^d)$ are insufficient to yield uniform mode of continuity for solutions of \eqref{main}. Our example used a series of potential vector fields with strong compression at one point, which yields discontinuity in limit.  In this section we will show that the loss of regularity continues to be true for divergence free vector fields, though here we are only able to present loss of H\"{o}lder estimates. Our example leaves open the possibility of weaker modulus of continuity.

Our examples are inspired by that of  \cite{loss}, where parallel results are shown for a fractional diffusion-drift equation, however there is a significant difference on the barrier argument that is presented below.  While \cite{loss} makes use of the nature of their fractional diffusion, we make use of the degeneracy of the diffusion in the small density zone.  More precisely, our counterexamples will describe loss of regularity near small density region, due to the discontinuities of the drifts across the cone $|x|=|y|$ (in two dimensions) and $|y| = |x_1 \pm x_2|$ (in three dimensions).  In the construction of barriers below, our use of degenerate diffusion appears both in the evolution of the density height in the barriers, and in the construction of supersolution where the zero set propagates with finite speed. See below for further discussion on the construction of barriers.


\medskip

\subsubsection{Example in $d=2$}


Let us begin with two dimensions, where the presentation illuminates the main components of the argument better. Let us recall that,  for $u$ solving \eqref{main} with divergence free $V$, the pressure variable $v:=\frac{m}{m-1}u^{m-1}$ solves
\begin{equation}\label{pressure2}
v_t-(m-1)v\Delta v-|\nabla v|^2+V\cdot \nabla v=0.
\end{equation}
We will prove the following theorem by constructing barriers for the pressure equation above.

\medskip

\begin{theorem}\label{divfree2}
There is a sequence of bounded vector fields $\{V_n\}$ which are uniformly bounded in $ L^2(\R^2)$, and a sequence $\{u_n\}_n$ of solutions for \eqref{pressure2} with $V_n$ that satisfies the following:
 \begin{itemize}
 \item[1.] $\{u_n(x,0)\}$ are uniformly bounded in $C^k(\R^2)$ for any $k>0$;
\item[2.]  $\{u_n\}$ are uniformly bounded in $\R^2\times [0,1]$;
\item[3.] For any $\delta>0$ we have
$\sup_n [ u_n]_{\delta}  = \infty$,
where $[f]_{\delta}$ denotes the $C^{\delta}$ semi-norm of $f$ in $\R^2\times [0,1]$.
\end{itemize}

\end{theorem}

$\circ$ {\it Construction of vector fields}

\medskip

For $s\in (0,1)$ define
\[\psi(x,y):=\frac{1}{2}s^\frac{1}{2}\left(|x-y|^s-|y+x|^s\right).\]
Let us also define smooth cut-off functions $\kappa$ and $\mu_\e$  satisfying
\begin{equation*}\label{chi2}\chi_{[-\frac{1}{3},\frac{1}{3}]}\leq \kappa \leq \chi_{[-\frac{1}{2},\frac{1}{2}]}\end{equation*}
and
\begin{equation}\label{muep2}\chi_{[2\epsilon,10)}\leq \mu_\epsilon \leq \chi_{[\epsilon,20)},\quad |\mu_\epsilon'|(x,y)\leq \frac{2}{|(x,y)|}.\end{equation}

Now define  
\begin{equation}\label{vector_field}
V:= \nabla^\perp F = \l-\partial_y F,\partial_x F\r, \hbox{ where }  F(x,y):=\psi(x,y) \kappa\left(\frac{x}{y}\right)\mu_\epsilon(|(x,y)|).
\end{equation}

We claim that for all $s, \e \in (0,1)$, $V$ is bounded uniformly in $L^2(\mathbb{R}^d)$. 
To see this,  note that by definition we have 
$$|\nabla^\perp F|=|\nabla^\perp ( \psi\kappa)|\leq Cs^\frac{1}{2}|y|^{s-1}\hbox{ for }|(x,y)|\in [2\epsilon,1].
$$ Then
\[ \left\|\nabla^\perp F(x,y)\right\|^2_{L^2(B_1\backslash B_{2\epsilon})}\leq C \int_0^1\int_0^{cy}s|y|^{2(s-1)}dxdy\leq C.\]
On the other hand in $B_{2\epsilon}$, by \eqref{muep2}, we have $|\nabla^\perp F|\lesssim x^\frac{1}{2}|(x,y)|^{s-1}$. The claim follows now from above computations and the truncation. 

\medskip



We will prove Theorem~\ref{divfree2} by comparison principle, Theorem~\ref{thm:comp}. More precisely, below we will construct a subsolution $\bar{u}$ and a supersolution $\ud{u}$  of \eqref{pressure2} to compare with $v_s$, to show the following:
\begin{Claim}
There exists a family of solutions $\{v_s\}_{s>0}$ of \eqref{pressure2}, with smooth initial data bounded uniformly in $C^2$ with respect to $\e$ and $s$, such that the following holds: 
\begin{itemize}
\item[1.]  $v_s(0,4\epsilon,T)\geq C_{s} \epsilon^{2s}$,
\item[2.]$v_s(0,-4\epsilon,T)=0$.
\end{itemize}
\end{Claim}
This claim will conclude the theorem if we take $s\to 0^+$.  

\begin{figure}[b]\centering\includegraphics[width=0.6\textwidth]{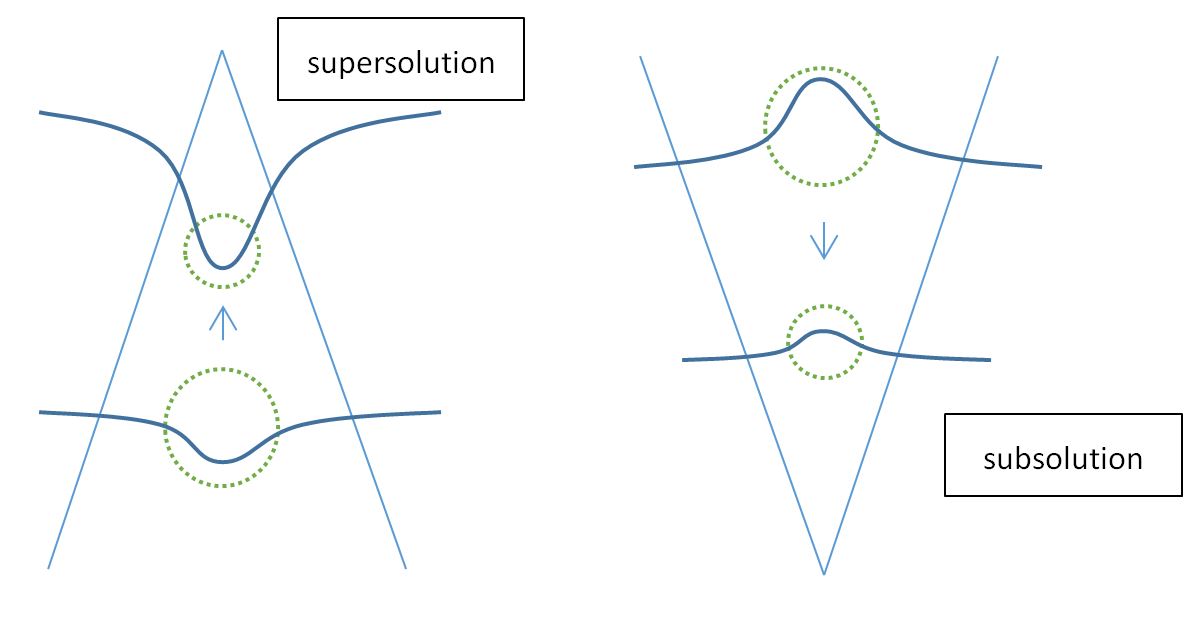}\label{figure1}\caption{Sub and Super Solutions}\end{figure}

\medskip

Roughly speaking, our barriers are of the form $\bar{u}(x,t) = k(t) \Phi_1(x,t)$ and $\ud{u}(x,t)= \tilde{k}(t)\Phi_2(x,t)$, where $k$ and $\tilde{k}$ are carefully chosen to estimate the evolution of density height inside and outside of the singular cones (see Figure 1). The spatial function $\Phi_1$ is a small bump function of height $1$ considered in \cite{loss}. While $\Phi_2$ roughy amounts to $1-\Phi_1$, it presents a nonnegative function with nonempty zero set that moves with finite speed. This is a consequence of the finite propagation property of degenerate diffusion, and is a crucial feature of $\ud{u}$ that is needed to establish the claim above.

\medskip

$\circ$ {\it Construction of subsolution}
 
 \medskip
 
 For  $\e>0$, let us define the parameters
\begin{equation}\label{TT}M:=s^{-\frac{3}{2}}\quad \text{ and }\quad T:=M(1-(4\epsilon)^{2-s})/(2-s).\end{equation} 
Define 
\begin{equation}\label{zz}z(t):=(1-(2-s)M^{-1}t)^\frac{1}{2-s}, \quad t\in[0,T]\end{equation}
so that $z$ satisfies
\[z'=-M^{-1}z^{s-1}, \quad z(0)=1,\quad z(T)=4\epsilon.\]

Let $\varphi\in C_0^\infty(B_1(0))$ be a smooth, non-negative, radially symmetric and decreasing function with the property $|\Delta\varphi|
\leq C$. 
Now for some $r\in(0,\frac{1}{8})$ and a constant $c_s>0$, define 
\[ \bar{u}((x,y),t):=c_{s}z^{s}(t)\Phi((x,y),t):=c_{s}z^s(t)\varphi\left(\frac{x,y-z(t)}{rz(t)}\right).\]
 We will choose $r$ small enough such that $\bar{u}$ is supported in the upper cone $y>3(|x|)$.

\begin{lemma}\label{sub2}
Let $\bar{u}$ be defined as above. Then there exists $r_s>0$ which depends on $s$ (but independent of $\epsilon$) such that for $r \leq r_s$ and $c_{s}$ small enough independent of $\epsilon$, $\bar{u}$ is a subsolution to \eqref{pressure2}.
\end{lemma}

\begin{proof}
For simplicity, let us omit the subscript $s$, and denote $z=z(t)$.
By definition we only need to check that
\[\bar{u}_t-(m-1)\bar{u}\Delta \bar{u}+V\cdot\nabla \bar{u}\leq 0\]
which is equivalent to
\[
c_{s}(( z^s)\Phi+z^s\partial_t\Phi)-(m-1)c^2_sz^{2s}\Phi\Delta \Phi+c_sz^s V\cdot\nabla\Phi\leq 0.
\]
Recall $|\Delta\varphi|\leq C$, it is sufficient to show that
\begin{equation}\label{suff2}(z^s)'\leq (m-1)c_s z^{2s}\Phi\Delta\Phi \leq -\frac{C(m-1)}{r^2z^2}c_{s}z^{2s} \hbox{ and } \end{equation}
\begin{equation}\label{suff12}\partial_t\Phi+V\cdot\nabla\Phi\leq 0.\end{equation}

The first inequality is equivalent to
 $$
C(m-1)c_{s}\leq sr^2M^{-1}= sr^2M^{-1}.
$$
So once $r=r_s$ and $c_s$ are chosen small enough that 
\begin{equation}\label{small}
c_{s}\leq \frac{sr^2}{C(m-1)M}\lesssim s^{\frac{5}{2}}r^2,
\end{equation}
inequality \eqref{suff2} holds.

Next we claim that there exists a universal constant $r_s>0$ which is independent of $\epsilon$ such that for all $r\leq r_s$, \eqref{suff12} holds. From the computation in Lemma 3.6 \cite{loss}, there exists a universal constant $r_s>0$ which is independent of $\epsilon$ that for all $r\leq r_s$, \eqref{suff12} holds. (Similar calculations will be performed for construction of a supersolution,  see Lemmas \ref{6.2} - \ref{6.5}). It follows that for $r$ and $c$ small enough, $\bar{u}$ is a subsolution and $\bar{u}(0,4\epsilon,T)\sim c_{s} z^s(T)=c_{s}(4\epsilon)^s$. Due to \eqref{small},  $\bar{u}(x,y,0)$ is uniformly $C^k$ for any $k$ with respect to $\e$. 

\end{proof} 


\medskip

$\circ$ { \it Construction of supersolution}

\medskip

Let $M,T,z(t)$ be as previously defined in \eqref{TT}, \eqref{zz}.
\medskip
 
 Let us consider a smooth function $\varphi(R): [0,\infty) \to [0,1] $ with the following properties:
\begin{itemize}
\item[1.]$\varphi$ is increasing with $\varphi(0)=\varphi'(0)=0$ and $\varphi\equiv 1$  for $R\geq 1$.
\item[2.]There exists a constant $C^*>0$ that 
\begin{equation}\label{barrier101}
(m-1)\varphi(\varphi'' +  \frac{1}{R}\varphi') +|\varphi'|^2\leq C^*\varphi
\end{equation}.
\end{itemize}
To construct such $\varphi$, for instance we can choose $\varphi = R^2$ for $|R|\leq 1/2$ and extend it to a smooth function satisfying 1,2. 
 With the above $\varphi$,  define 
$$
\Phi((x,y),t):=\varphi\left(\Big|\frac{(x,y+z(t))}{rz(t)}\Big|\right).
$$
and
\begin{equation}\label{kk} k(t):=\left(C_0-C^*\frac{M r^{-2}}{s}z(t)^{s}\right)^{-1},
\end{equation}
where 
\begin{equation}\label{kk2}
C_0:=2C^*Mr^{-2}s^{-1}(4\epsilon)^{-s}\sim r^{-2} s^{-\frac{5}{2}}\epsilon^{-s}.
\end{equation}
The choice of $C_0$ is to ensure that $k(t)$ stays nonnegative for $0\leq t\leq T$ and $k(0)=\frac{2}{C_0}= \frac{4^s}{C^*}r^{2}s^{\frac{5}{2}}\epsilon^s$.

\medskip

 With the functions $k$ and $\Phi$ as defined above, we will consider a supersolution of the form
\begin{equation}\label{supersolution}
\ud{u}(x,y,t):=k(t)\Phi(x,y,t).
\end{equation}

\begin{lemma}\label{6.2}
Let $\ud{u}(x,y,t)$ be as given in \eqref{supersolution}. There exist $r_s>0$ independent of $\epsilon$ and a universal constant $C_1>0$. If $r\leq r_s$ and $k(0)\leq C_1 r^2s^{\frac{5}{2}}\epsilon^s$, then $\ud{u}$ is a supersolution to \eqref{pressure2} in the time interval $[0,T]$. 
\end{lemma}
 
\begin{proof}

Consider the region $S:=\{(x,y),\; 3\epsilon\leq 3|x|\leq -y\}\times [0,T]$. Showing that $\ud{u}$ is a supersolution is equivalent to 
\[
( k'\Phi+k\partial_t\Phi)-(m-1)k^2\Phi\Delta \Phi-k^2|\nabla\Phi|^2+k V\cdot\nabla\Phi\geq 0.
\]
By \eqref{barrier101},
\[Ck^2z^{-2}r^{-2}\Phi\geq (m-1)k^2\Phi\Delta \Phi+k^2|\nabla\Phi|^2.\]
Thus it suffices to show
\begin{equation}\label{ineq2}
k'\geq  Ck^2z^{-2}r^{-2} \hbox{ and } \partial_t\Phi+V\cdot\nabla\Phi\geq 0\;
\end{equation}

The first inequality in \eqref{ineq2} follows from the construction of $k(t)$. The second inequality can be written as 
\[\nabla\varphi\cdot (x,y)\;z^{-1-s}+\nabla\varphi\cdot (M V z^{-1})\geq 0.\]
 Notice that inside $S$,  $\tilde{V}:= MV$ satisfies
\[\left(\tilde{V}_1,\tilde{V}_2\right)=\frac{1}{2}\left(|x-y|^{s-1}-|-x-y|^{s-1},\;|x-y|^{s-1}+|-x-y|^{s-1}\right).\]
Since $\nabla \varphi$ is in the direction of $(x,y+z)$, the above inequality is equivalent to
\[M(x,y,z):=(x\;,y+z)\cdot (x,y)z^{s-2}+(x\;,y+z)\cdot (\tilde{V_1},\tilde{V_2})\geq 0.
\]
By $(s-1)$-homogeneity of $(\tilde{V}_1,\tilde{V}_2)$, it is then equivalent to verify for $|x|^2+|y+1|^2\leq r^2$, we have
\[f(x,y):=M(x,y,1)= x^2+y(y+1)+\frac{1}{2}|x-y|^
{s-1}(x+y+1)+\frac{1}{2}|-x-y|^{s-1}(-x+y+1)\geq 0.\]
After basic computations
\[f(0,-1)=f_x(0,-1)=f_y(0,-1)=f_{xy}(0,-1)=0,
\]
\[f_{xx}(0,-1)=2s>0,\; f_{yy}(0,-1)=2(2-s)>0.\]
It follows  that $(0,-1)$ is a local minimum of $f$ and therefore there exists $r_s$ such that $f\geq 0$ inside $|x|^2+|y+1|^2\leq r_s^2$. And hence the first inequality of \eqref{ineq2} holds.

\medskip

It follows that $\ud{u}$ is a supersolution in $\mathbb{R}^2\times [0,T]$ with $\ud{u}(x,y,0)\leq C_1r_s^{2}s^{\frac{5}{2}}\epsilon^s$ in the lower half plane ($y\leq 0$).
Since $\varphi(0,0)=0$ at time $T$ we have
\[\ud{u}(0,-4\epsilon,T)=0.\]
\end{proof}


{\bf Proof of Theorem~\ref{divfree2}: }

Now for any module of holder continuity $w(\tau)=C\tau^{\delta}$. Let us select $s<\frac{1}{2}\delta$ and $\epsilon$ arbitrarily small. 
Let $r$ be sufficiently small so that Lemma \ref{sub2} and Lemma \ref{6.2} applies. Let $C_1$ be the constant in Lemma \ref{6.2}.

Consider a smooth function $v_0:\R^d\to \R$ be supported in the upper half plane and
\begin{itemize}
\item[1.] $v_0\geq \frac{C_1}{2} r_s^{2}s^{\frac{5}{2}}\epsilon^s$ in $B_{r_s}(0,1)$;
\item[2.] $v_0\leq C_1r_s^{2} s^{\frac{5}{2}}\epsilon^s$.
\end{itemize}


Let $v_s = v_{s,\e}$ solve \eqref{pressure2} with initial data $v_0$. Let us choose $\epsilon$ small enough so that $\bar{u}$ given in Lemma \ref{sub2} with $c_s:=  \frac{C_1}{2} r_s^{2}s^{\frac{5}{2}}\epsilon^s$ is a subsolution of \eqref{pressure2}.
From comparison principle, the solution to \eqref{pressure2} with initial data $v_0$ satisfies 
\begin{equation}\label{111}
v(0,4\epsilon,T)\geq Cr_s^2s^{\frac{5}{2}} \epsilon^{2s}.
\end{equation}

Next let $\ud{u}$ be the supersolution as given in Lemma \ref{6.2}. Then we have $\ud{u}(\cdot,0)\geq v_0$, thus by comparison principle it follows that $v_2 \geq v$. 
Then at time $T$,
\begin{equation}\label{222}
v_s(0,-4\epsilon,T)\leq \ud{u}(0,-4\epsilon,T)=0.
\end{equation}
Putting \eqref{111} and \eqref{222} together, it follows that 
\[|v_s(0,4\epsilon,T)-v_s(0,-4\epsilon,T)|/|8\epsilon|^\delta\geq   Cr_s^2s^{\frac{5}{2}} \epsilon^{2s-\delta}  = C(s)\epsilon^{2s-\delta}.
\]

\medskip

Finally, let us normalize parameters so that the singular time $T$ is comparable to $1$. 
\[
u_{s,\e}(x,t)=v_{s,\e}(M^\frac{1}{2}x,Mt).
\]
Let us  normalize $T$  by 
$$
\tilde{T}=T/M=\frac{1-(4\epsilon)^{2-s}}{2-s},
$$ which is close to $1/2$ for all $s,\epsilon$ close to $0$.
Recall $V$ defined in \eqref{vector_field}. Then $u_{s,\e}$ solves equation \eqref{pressure2} with $V$ replaced by 
\[\tilde{V}(x) = \tilde{V}_{s}:=M^\frac{1}{2}V(M^\frac{1}{2}x),\]
where $V$ is defined in \eqref{vector_field}. Then $\{\tilde{V}_s\}$ are uniformly bounded in $L^2(\R)$ for all $s$. 

From Theorem \ref{divfree2},
\[|u_s(0,4\epsilon/M^\frac{1}{2},\tilde{T})-u_s(0,-4\epsilon/M^\frac{1}{2},\tilde{T})|/|8\epsilon|^\delta\geq C_s\epsilon^{2s-\delta}.\]
Then as $\epsilon\to 0$, any $C^\delta$- norm with $\delta>2s$ again grows to infinity at time $\tilde{T}$ which is uniformly bounded this time. Thus we can conclude our theorem if we choose 
$$
u_n:= u_{1/n, \e_n}. 
$$
where $\e_n$ is chosen sufficiently small such that the $C^n$ norm of  $u_{1/n,\e_n}(x,0)$ is bounded.

\hfill$\Box$

\subsubsection{Example in $d=3$}

\begin{theorem}\label{divfree}
There exist a sequence of bounded vector fields $\{V_n\}$ which are uniformly bounded in $ L^3(\R^3)$ such that parallel statements as in \eqref{divfree2} holds.\end{theorem}

$\circ$ {\it Construction of vector fields}

\medskip

Let us denote $x=(x_1,x_2,y)\in\mathbb{R}^3$. For $s\in(0,1)$, define 
$$\psi(x_1,x_2,y):=\left(-(y+x_1-x_2)^{s}+(y+x_1+x_2)^{s},(y-x_1+x_2)^{s}-(y+x_1+x_2)^{s},\;0\;\right).
$$
For $\e>0$, let $\kappa$ and $\mu_\e$ be two smooth cut-off functions satisfying
\begin{equation}\label{chi}\chi_{[-\frac{1}{3},\frac{1}{3}]}\leq \kappa \leq \chi_{[-\frac{1}{2},\frac{1}{2}]}\end{equation}
and
\begin{equation}\label{muep}\chi_{[2\epsilon,10]}\leq \mu_\epsilon \leq \chi_{[\epsilon,20]},\quad |\mu_\epsilon'|(x)\leq \frac{2}{|x|}.\end{equation}
Now we define $V:= \nabla\times F$ with  \[F(x):=\frac{1}{4}({s})^{\frac{1}{3}}\psi(x) \kappa\left(\frac{x_1-x_2}{y}\right)\kappa\left(\frac{x_1+x_2}{y}\right)\mu_\epsilon(|x|).\]

We claim that for all $s,\epsilon\in (0,1)$ any small, $V$ is bounded uniformly in $L^3(\R^3)$. To show this, by symmetry it is enough to consider the following regions:
\[S_1:=\left\{(x_1,x_2,y)\in B_1,\;y\geq 3\max\{|x_1+x_2|,|x_1-x_2|\}, |(x_1,x_2,y)|\geq 2\epsilon\right\},\] 
\[S_2:=\left\{(x_1,x_2,y)\in B_1,\;
\frac{1}{2}y\geq x_1+x_2\geq \frac{1}{3}y>0, |(x_1,x_2,y)|\geq 2\epsilon\}\right\},\]
\[S_3:=\left\{(x_1,x_2,y)\in B_1,\; |(x_1,x_2,y)|\leq 2\epsilon\right\}.\] 

\medskip

In $S_1$, $\kappa=\mu_\epsilon=1$ and $\|\nabla\times\psi\|\leq Cs|y|^{s-1}$. Therefore
\[ \left \| V(x)\right\|^3_{L^3(S_1)}\leq C \int_0^1\int_0^{\frac{1}{3}y}\int_0^{\frac{1}{3}y}s^4|y|^{3(s-1)}dx_1dx_2dy\leq s^3C.\]
In $S_2$, since $|\kappa'|,|\kappa|$ are bounded and $\mu_\epsilon=1$, each component in $\nabla\times F$ is bounded by $Cs^\frac{1}{3}|y|^{s-1}$. Since $\mu_\epsilon'\leq \frac{2}{|x|}$ similar bound holds in $S_3$, and we have
\[ \left\|V(x)\right\|^3_{L^3(S_2\cup S_3)}\leq C \int_0^1\int_0^{\frac{1}{2}y}\int_0^{\frac{1}{2}y}s|y|^{3(s-1)}dx_1dx_2dy\leq C.\]

As before, we will prove Theorem~\ref{divfree} by comparison principle.

Take $ M=s^{-\frac{4}{3}}$. We define $T,z(t)$ the same as in \eqref{TT}, \eqref{zz}.
We can write $V=s^\frac{4}{3}(V_1,V_2,V_3)$, inside $S_1$ 

\begin{align*}
V_1&=-\frac{1}{4}(y-x_1+x_2)^{s-1}+\frac{1}{4}(y+x_1+x_2)^{s-1},\\ V_2&=-\frac{1}{4}(y+x_1-x_2)^{s-1}+\frac{1}{4}(y+x_1+x_2)^{s-1}, \\
V_3&=-\frac{1}{4}(y-x_1+x_2)^{s-1}-\frac{1}{4}(y+x_1-x_2)^{s-1}-\frac{1}{2}(y+x_1+x_2)^{s-1}.\end{align*}\medskip

$\circ$ {\it Construction of subsolution}

\medskip

Let $\varphi\in C_0^\infty(B_1(0))$ be a smooth, non-negative, radially symmetric and decreasing function with $|\Delta\varphi|
\leq C$ for some $C>0$. 
For $r\in(0,\frac{1}{9})$ and a constant $c_s$, define
\[ \bar{u}(x,t):=c_{s}z^{s}(t)\Phi(x,t):=c_{s}z^s(t)\varphi\left(\frac{x_1,x_2,y-z(t)}{rz(t)}\right).\]
Then the support of $\bar{u}$ lies inside the upper cone $S_1$. 

\begin{lemma}\label{sub}
Let $\bar{u}$ be defined as above. Then there exists $r_s>0$ independent of $\e$ and a universal constant $C>0$  such that for $r\leq r_s$ and $c_{s} = Cs^\frac{7}{3}r^2$, $\bar{u}$ is a subsolution to \eqref{pressure2}. Furthermore $\bar{u}(0,0,4\epsilon,T)\geq c_{s}(4\epsilon)^s$.  
\end{lemma}

\begin{proof}

We need to check that
\[\bar{u}_t-(m-1)\bar{u}\Delta \bar{u}+V\cdot\nabla \bar{u}\leq 0\]
inside the support of $\bar{u}$, which lies in $B_{rz}$.
Since $|\Delta\varphi|\leq C$, it suffices to show that
\begin{equation}\label{suff}(z^s)'\leq -\frac{C(m-1)}{r^2z^2}c_{s}z^{2s} \end{equation}
and
\begin{equation}\label{suff1}\partial_t\Phi+V\nabla\Phi\leq 0\end{equation}
in $B_{rz}$.

Since \eqref{suff} is equivalent to $C(m-1)c_{s}\leq sr^2M^{-1},$
it holds when 
\[c_{s}:=\frac{sr^2}{C(m-1)M}\lesssim s^{\frac{7}{3}}r^2.\]

Next notice
\[\partial_t\Phi+\tilde{V}\cdot\nabla\Phi=0, \text{ with }\tilde{V}=-M^{-1}z^{s-2}(x_1,x_2,y).\]
Hence to show \eqref{suff1}, it suffices to show $(\tilde{V}-V)\cdot\nabla\Phi\geq 0$ for $t\in[0,T] $ and for $(x_1,x_2,y-z)\in B_{rz}$. 

\medskip

Recall that $V=s^\frac{4}{3}(V_1,V_2,V_3)$, $M=s^{-\frac{4}{3}}$. Since $\nabla \Phi$ is parallel to $(x_1,x_2,y)$, it suffices to show that
\[((V_1,V_2,V_3)+z^{s-2}(x_1,x_2,y))\cdot (x_1,x_2,y-z)\geq 0 \hbox{ for } \{x: x_1^2+x_2^2+(y-z)^2\leq z^2r^2\}.\]
By $(s-1)$-homogeneity of $V$, this is equivalent to
\begin{equation}\label{rhsxy}( (V_1,V_2,V_3)+(x_1,x_2,y))\cdot(x_1,x_2,y-1)\geq 0 \hbox{ for } \{x: x_1^2+x_2^2+(y-1)^2\leq r^2\},
\end{equation}

The left handside of \eqref{rhsxy} can be written as 
\[f(x_1,x_2,y)=-\frac{1}{4}|y-x_1+x_2|^{s-1}(x_1+y-1)-\frac{1}{4}|y+x_1-x_2|^{s-1}(x_2+y-1)\]\begin{equation}\label{ff}-\frac{1}{4}|y+x_1+x_2|^{s-1}(-x_1-x_2+2(y-1))+x_1^2+x_2^2+y(y-1).
\end{equation} Straightforward computation yields
\[f(0,0,1)=f_{x_1}(0,0,1)=f_{x_2}(0,0,1)=f_{y}(0,0,1),\]
\[f_{x_ix_i}=1+s, f_{yy}=4-2s, f_{x_1x_2}=0, f_{x_iy}=-\frac{1}{2}(s-1).\]
So $(0,0,1)$ is a local minimum of $f$. Hence there exists $r_s>0$ which only depends on $s$ such that \eqref{rhsxy} holds for $(x_1,x_2,y-z)\in B_{r_sz}$, thus we conclude that $v$ is a subsolution of \eqref{pressure2} when $c_s$ and $r_s$ are sufficiently small. In particular observe that 

$$\bar{u}(0,0,4\epsilon,T)\sim c_{s} z^s(T)=c_{s}(4\epsilon)^s \hbox{ with }c_{s}\leq Cs^\frac{7}{3}r^2,
$$ where $C$ is a universal constant which {is independent of $s,\epsilon$.} 

\end{proof}

\medskip

$\circ$ {\it Construction of supersolution}

\medskip

Let $\varphi: [0,\infty) \to [0,1] $ be  as given in \eqref{barrier101} and let $k(t)$ be as given in \eqref{kk}, \eqref{kk2} with $M=s^{-\frac{4}{3}}$.  Recall that \begin{equation}\label{k'C}
k'\geq Cz^{-2}r^{-2}k^2 \text{ for }t\in[0,T].
\end{equation}
 For $r<\frac{1}{9}$, we define
$$
\ud{u}(x,t)=k(t)\Phi(x,t):=k(t)\varphi\l \frac{(x_1,x_2,y+z(t))}{r z(t)}\r.
$$

\begin{lemma}\label{6.5}
Let $\ud{u}$ be defined as above, and let $r_s$ as given in Lemma \ref{6.2}. If $r\leq r_s$, $\ud{u}$ is a supersolution to \eqref{pressure2} in $\R^3\times [0,T]$. Furthermore $\ud{u}(0,0,-4\epsilon,T)=0$ and
$\ud{u}(x,0)\geq C r^2s^\frac{7}{3}\epsilon^s$, where $C$ is independent of $s$ and $\e$. 
\end{lemma}

\begin{proof}
As done in Lemma \ref{6.2}, it is sufficient to show
\begin{equation}
k'\Phi\geq  Cz^{-2}r^{-2}k^2\varphi,\label{suffi}\quad \text{ and }\quad
\partial_t\Phi+V\cdot\nabla\Phi\geq 0
\end{equation}
in $S:=\left\{(x_1,x_2,y),\; -y\geq\max\{|x_1+x_2|,|x_1-x_2|\},\; |x|\geq 2\epsilon\right\}$.

\medskip

The first inequality in \eqref{suffi} holds, as before, due to \eqref{barrier101} and the definition of $k(t)$. To show the second inequality, write $V=M^{-1}(V_1,V_2,V_3)$. In $y<0$ 
$$
M\left(\partial_t\Phi+V\cdot\nabla\Phi\right)=\nabla\varphi\cdot (x_1,x_2,y+z) z^{-3+s} +\nabla\varphi\cdot (V_1,V_2,V_3) z^{-1}.
$$
By definition in the region $S$ 
\begin{align*}
V_1&=\frac{1}{4}(y-x_1+x_2)^{ {s-1}}-\frac{1}{4}(y+x_1+x_2)^{ {s-1}},\\ V_2&=\frac{1}{4}(y+x_1-x_2)^{ {s-1}}-\frac{1}{4}(y+x_1+x_2)^{ {s-1}}, \\
V_3&=\frac{1}{4}(y-x_1+x_2)^{ {s-1}}+\frac{1}{4}(y+x_1-x_2)^{ {s-1}}+\frac{1}{2}(y+x_1+x_2)^{ {s-1}}.\end{align*}

As before we only need to verify that there exists $r=r_s$ such that inside $|x_1|^2+|x_2|^2+|y+1|^2\leq r$
\[|x_1|^2+|x_2|^2+(y+1)y+\frac{1}{4}|y-x_1+x_2|^{ {s-1}}(x_1+y+1)+\frac{1}{4}|y+x_1-x_2|^{ {s-1}}(x_2+y+1)\]\[+\frac{1}{4}|y+x_1+x_2|^{ {s-1}}(-x_2-x_1+2y+2)\geq 0.\]
Recall $f$ defined in \eqref{ff}, then the above is equivalent to
\[f(-x_1,-x_2,-y)\geq 0\]
near $(0,0,-1)$ which has already been verified when $r$ is small enough (depending only on $s$). Hence $\ud{u}$ is a supersolution. Note that $k(0)\geq C r^2 s^{\frac{7}{3}}\epsilon^{s} $, and thus we conclude.

\end{proof}

\begin{proof}of Theorem~\ref{divfree}. The proof is now parallel to the two dimensional case Theorem \ref{divfree2}, with the help of Lemma \ref{sub} and Lemma \ref{6.5}.

\end{proof}

\appendix

\section{Proof of Lemma \ref{iteration}}

Suppose $B_k(0),B_0(t)$ are bounded by $M$, then $A_k(0)\leq M^{n_k}$. Solving the differential inequality gives that for all $t\geq 0$
\[A_k(t)\leq e^{-C_0t}\int_{0}^te^{C_0 s}\l {C_1}^{n_k}+{C_1}^k A_{k-1}^{2+{C_1}n_k^{-1}}(s)\r ds+M^{n_k}.\]
If $A_{k-1}(t)$ are uniformly bounded by $M_{k-1}$ for all $t$, we can choose a constant $C_2$ depending only on $(C_0,C_1,M)$ such that
\begin{equation}\label{Ak}A_k(t)\leq CC_1^{n_k}+CC_1^k M_{k-1}^{2+C_1 n_k^{-1}}+M^{n_k}\leq C_2^{n_k}+{C_2}^k M_{k-1}^{2+{C_1}n_k^{-1}}.\end{equation}
We claim that it can be proved by induction that 
\begin{equation}\label{ineq:Ak}A_k(t)\leq C_3^{c_k} \text{ for some constants }C_3(C_0,C_1,M),c_k(C_1,k).\end{equation}
Here $\{c_k\}$ is defined inductively by
\begin{equation}\label{defc}c_0=1,\quad c_{k}:=(2+\frac{C_1}{n_k})c_{k-1}+k+1.\end{equation}
By a slight abuse of notation, we will write $C$'s as constants which only depend on $C_1,C_0,M,a$ (independent of $k$) and they may vary from one expression to the other. 

To see the claim, by induction taking $M_{k-1}=C_3^{c_{k-1}}$ in \eqref{Ak}, we only need
\[C_2^{n_k}+C_2^kC_3^{c_{k-1}({2+C_1n_k^{-1}})}\leq C_3^{c_k}.\]
And it is not hard to see by definition, $n_k\lesssim k+c_{k-1}({2+C_1n_k^{-1}})$. So if choosing $C_3$ large enough, we only need
\[C_3^{k+1+c_{k-1}({2+C_1n_k^{-1}})}\leq C_3^{c_k}\]
which is exactly \eqref{defc}. We proved the claim.

By \eqref{defc} and simple calculations, 
\[c_k=\left(\sum_{j=1}^k j\; b_{j,k}\right)+k+1\quad\text{ where }
b_{j,k}:=\Pi_{i=j}^{k}(2+\frac{C_1}{n_i}).\]
Notice $n_k=2^k(a+1)-a$, there is a constant $C_4(a,C_1)$ that $\frac{C_1}{n_k}\leq C_42^{-k}$ for all $k\geq 0$. So
\[b_{j,k}\leq 2^{k-j+1} \Pi_{i=j}^{k}(1+C_42^{-k-1}).\]
Then we apply the fact that given $x_n\geq 0 $ and $\sum_nx_n\leq C_4$, we have for some other constant $C>0$
$$\Pi_n (1+x_n)\leq  C+C\sum_n x_n.$$
We find out
\[b_{j,k}\leq 2^{k-j+1}C(1+C_4 )\lesssim 2^{k-j} \quad \text{ and }\]
\[c_k\leq  C2^k\sum_{j=1}^k \frac{j}{2^j}+k+1\lesssim 2^{k}\lesssim n_k .\]
So $c_k\leq Cn_k$. By \eqref{ineq:Ak}, we proved that 
\[A_k^{(n_k^{-1})}(t)\leq C_2^{C}\quad \text{ uniformly for all } k\in\mathbb{N}^0\text{ and }t\geq 0.\]

\section{Proof of Lemma \ref{twoseq}}

The idea of the proof is to find a finite $N(p,d),C(N),\epsilon>0$ such that for all natural numbers $k\geq 0$
\begin{equation}\label{ineq:appb}a_{N(k+1)}\leq C^{k}a_{Nk}^{1+\epsilon}.\end{equation}
Then the proof again follows from the iteration, see Lemma 2.2 \cite{dibenedetto1}. We claim that this can be done by simply  plugging the first inequality into the second one for finite times. If we do it once
\[ a_{n+1}\leq C_1^n a_n^{q_2+\frac{2}{d+2}}+C_1^{n} a_n^\frac{2}{d+2}
\l C_1^{n-1} a_{n-1}^{q_2+\frac{2}{d}}+C_1^{n-1} a_{n-1}^\frac{2}{d}b_{n-1}^{q_1}\r^{q_1}.\]
Recall $q_1=1-\frac{2}{p},q_2=1-\frac{1}{p}$. Suppose $n=2k+1$ and from the above we have
\begin{align*}
a_{2(k+1)}&\leq C_1^{2k+1} a_{2k+1}^{q_2+\frac{2}{d+2}}+C_1^{4k+1} a_{2k+1}^{\frac{2}{d+2}}a_{2k}^{q_1q_2+\frac{2q_1}{d+2}}+C_1^{4k+1} a_{2k}^{\frac{2}{d+2}+\frac{2q_1}{d}}b_{2k}^{q_1^2}\\
&\leq C^k\l  a_{2k}^{q_2+\frac{2}{d+2}}+ a_{2k}^{\frac{2}{d+2}+q_1q_2+\frac{2q_1}{d+2}}+ a_{2k}^{\frac{2}{d+2}+\frac{2q_1}{d}}b_{2k}^{q_1^2}\r.
\end{align*}
We used $q_1,q_2<1$ in the first inequality and $a_{2k+1}\leq a_{2k}$ in the second one.
Also since $b_{2k}\leq a_{2k}$, if the following two inequalities hold:
\[\frac{2}{d+2}+(\frac{2}{d}+q_2)q_1>1\;\text{ and }\frac{2}{d+2}+\l\frac{2}{d}+q_1\r q_1>1,\]
\eqref{ineq:appb} holds and we finish the proof. Because $q_2>q_1$, we only need the second inequality to be true.

If the second one fails, then we do one more iteration. Similar computations and arguments imply that we then only need
\[\frac{2}{d+2}+\l\frac{2}{d}+\l\frac{2}{d}+q_1\r q_1\r q_1>1
\]
to be true in order to have \eqref{ineq:appb}.

If we keep doing the process, eventually, we want
\[\frac{2}{d+2}+\underbrace{\l\frac{2}{d}+\l\frac{2}{d}+...\l\frac{2}{d}+q_1\r q_1\r  q_1\r q_1}_{n \text{ brackets }}>1\;\text{ for some }n,\]
which is
\[\frac{2}{d+2}+\frac{2}{d}\l q_1+q_1^2+...+q_1^{n}\r+q_1^{n+1}>1 \;\text{ for some }n.\]
Letting $n=\infty$ gives
\[\frac{2}{d+2}+\frac{2}{d}\frac{q_1}{1-q_1}=\frac{2}{d+2}+\frac{2}{d}\frac{p-2}{2}>1\]
which is equivalent to
\[p>d+\frac{4}{d+2}\]
and we finish the proof of the lemma.


\end{document}